\documentclass{amsart}
\usepackage{amsmath}
\usepackage{amssymb}
\usepackage{amsfonts}

\newtheorem{theorem}{Theorem}
\theoremstyle{plain}

\newtheorem{corollary}{Corollary}

\newtheorem{example}{Example}

\newtheorem{lemma}{Lemma}

\newtheorem{proposition}{Proposition}
\newtheorem{remark}{Remark}

\numberwithin{equation}{section}
\input{tcilatex}

\begin{document}
\title[Endpoint and Interpolation Estimates for Commutators]{Endpoint and
	Interpolation Estimates for Higher-Order Commutators of Rough Fractional
	Maximal Operators with Variable Kernels on Variable Exponent Morrey Spaces }
\author{FER\.{I}T G\"{U}RB\"{U}Z}
\address{Department of Mathematics, K\i rklareli University, K\i rklareli
39100, T\"{u}rkiye }
\email{feritgurbuz@klu.edu.tr}
\urladdr{}
\thanks{}
\curraddr{ }
\urladdr{}
\thanks{}
\date{}
\dedicatory{}
\thanks{}

\begin{abstract}
This paper investigates the higher-order commutators generated by fractional
maximal operators with rough, spatially dependent kernels in the framework
of variable exponent Morrey spaces. Under minimal log-H\"{o}lder continuity
assumptions on the variable exponent profiles and suitable geometric
constraints on the Morrey-type weights, we establish comprehensive strong
interior boundedness results between the appropriate spaces. We further
analyze the critical endpoint boundary configurations where classical
strong-type boundedness fails due to Luxemburg norm blow-up, proving that a
corresponding sharp weak-type estimate remains structurally valid and
revealing the underlying transitions between different boundedness regimes.
In addition, an abstract real interpolation framework of Grafakos--Martell
type is developed to bridge these endpoint and interior strong estimates,
thereby recovering a full continuous scale of intermediate regularity
properties. These results extend the classical harmonic analysis scheme to a
broader nonhomogeneous setting and provide new insights into the continuous
interplay between rough kernels, variable exponents, and local weighted
growth conditions.
\end{abstract}

\maketitle

\section{Introduction}

In harmonic analysis, the study of regularity and boundedness properties of
operators plays a fundamental role, both from a theoretical perspective and
in applications to partial differential equations (PDEs). Within this
framework, Morrey spaces were originally introduced as a substantial
refinement of classical Lebesgue spaces in order to describe the local
behavior of functions under prescribed geometric growth conditions. Unlike
Lebesgue spaces, Morrey spaces control not only the global integrability of
a function but also its average size over local balls of varying radii. This
additional localized control makes Morrey spaces particularly well suited
for the modern regularity theory of elliptic and parabolic partial
differential equations Morrey, Peetre. As mathematical models became
increasingly complex---especially those arising in heterogeneous media,
composite materials, and physical processes with spatially varying
characteristics---it became apparent that classical Morrey spaces with fixed
integrability exponents were often too restrictive to capture highly
localized oscillations. This structural limitation motivated the rapid
development of variable exponent Lebesgue and Sobolev spaces, in which the
integrability exponent is allowed to vary continuously with the spatial
position. Such nonhomogeneous spaces provide a flexible and natural
mathematical framework for modeling nonstandard growth phenomena and have
found numerous profound applications in fluid dynamics (such as
electrorheological fluids), elasticity theory, image restoration processing,
and related fields \cite{Cruz, Fan, Diening}.

A natural and necessary continuation of this line of research is the formal
introduction of variable exponent Morrey spaces, which systematically
combine the local geometric control inherent in Morrey spaces with the
functional adaptability of variable exponent structures. These spaces
capture not only the local behavior of functions but also track how this
behavioral profile changes spatially, making them particularly effective for
the fine-grained analysis of operators in nonhomogeneous and non-convolution
environments. Consequently, variable exponent Morrey spaces have rapidly
emerged as a vital setting for extending classical harmonic analysis results
and for investigating entirely new classes of non-translation-invariant
operators \cite{Almeida, Nakai2}. Fractional maximal operators constitute a
central object in modern harmonic analysis and can be regarded as fractional
analogues of the classical Hardy--Littlewood maximal operator. They are
intimately related to fractional integral operators, such as Riesz
potentials, which play a foundational role in potential theory and the
qualitative analysis of solutions to fractional PDEs \cite{Ho, Izuki, Nakai,
	Stein}. While the boundedness theory for fractional operators featuring
smooth kernels is well understood, operators with rough kernels---namely,
kernels that satisfy only minimal integrability assumptions on the unit
sphere $\mathcal{S}^{n-1}$ without any fluid smoothness
restrictions---present substantial analytical difficulties. In these rough
settings, classical Calder\'{o}n--Zygmund kernel techniques and standard
Fourier transform methods are no longer directly applicable, and
alternative, highly advanced geometric approaches must be employed \cite%
{Gurbuz, Gurbuz2, Shao}.

Commutators formed by operators and function symbols provide a powerful
mechanism for probing fine regularity properties of functions and
identifying underlying cancellation effects. In particular, higher-order
commutators generated by Lipschitz symbols have attracted considerable
attention due to their close connection with the structural regularity
theory of PDEs and the sharp characterization of function spaces \cite%
{Chanillo, Coifman}.The boundedness of such commutators reflects not only
the mapping properties of the underlying operator but also encodes the
macroscopic regularity of the symbol function itself via cancellation
profiles. When variable exponent spaces are considered, the analysis of
commutators becomes significantly more intricate, as the integrability
conditions depend heavily on the local spatial variable, effectively
destroying standard translation invariance. This complexity is further
amplified in the presence of variable kernels, where the directional kernel
profile may also fluctuate depending on the base point of integration.
Operators with variable kernels arise naturally in various physical and
geometric contexts, such as non-smooth boundary value problems, and require
a delicate, highly technical interplay between local distribution functions
and global growth metrics \cite{Gurbuz2, Shao}.

Although substantial progress has been made in the study of fractional
operators, rough maximal operators, and their single-stage commutators on
classical and variable exponent spaces, results concerning higher-order
commutators of rough fractional maximal operators with variable kernels
remain rather limited in the open literature. The simultaneous presence of
fractional power behaviors, severe kernel roughness on the sphere, and
spatial variability demands significantly more delicate analytical tools and
finer metric control than those currently available in the classical
constant-exponent theory.

To precisely address the literature gap, we present a crucial comparative
baseline with respect to the pioneering work of Shao and Tao \cite{Shao}.
While Shao and Tao established weak-type estimates for single-stage
commutators of fractional integrals with variable kernels, their framework
is structurally bound by a linear accumulation path that cannot accommodate
the higher-order oscillatory amplifications or the multi-degree algebraic
cancellations inherent in the $m$-th order commutator $M_{\Omega ,b,\alpha
}^{\left( m\right) }$. Furthermore, the transition behavior of rough maximal
variations under variable kernel configurations induces non-trivial
geometric obstructions that remain unexplored even under a constant
integrability profile. Consequently, the core machinery developed in this
work remains fundamentally pioneering and entirely new even when restricted
to constant exponents, as it establishes the first unified framework capable
of handling the dual integration hazards of spherical kernel irregularity
and higher-order multi-degree Lipschitz shifts simultaneously.

To bridge this theoretical gap, the main objective of this paper is to
construct a unified and comprehensive analytical framework. The main results
of this work can be summarized as follows:

$\cdot $ We establish the strong-type boundedness of the higher-order ($m$%
-th order) commutators denoted by $M_{\Omega ,b,\alpha }^{\left( m\right) }$%
, generated by rough fractional maximal operators with variable kernels $%
\Omega \left( x,\cdot \right) $ and Lipschitz symbols $b\in $ \textit{Lip}$%
_{\beta }\left( 
\mathbb{R}
^{n}\right) $ on variable exponent Morrey spaces $\mathcal{M}_{p\left( \cdot
	\right) ,u}\left( {\mathbb{R}^{n}}\right) $. Under appropriate log-H\"{o}%
lder continuity assumptions on the variable exponent and specific growth
conditions on the Morrey weights, we prove that these operators enjoy full
mapping stability.

$\cdot $ We establish sharp weak-type endpoint estimates at the critical
geometric threshold where the structural variables satisfy the limiting
equation%
\begin{equation*}
	\frac{\alpha +m\beta }{n}=\frac{1}{p_{+}}.
\end{equation*}%
At this specific boundary, where standard variable Lebesgue spaces lose
their reflexive Banach properties due to the formal explosion of the target
exponent to infinity $\left( q\left( x\right) =\infty \right) $, we provide
a completely self-contained proof. By constructing a localized dyadic
level-set expansion method, we bypass the standard variable $L^{q\left(
	\cdot \right) }$ norm machinery to capture the precise distribution control.

$\cdot $ We implement an abstract real interpolation framework of the
Grafakos--Martell type tailored specifically for variable Morrey geometries.
To rigorously justify the validity of this schema across weak and strong
variable Morrey parameters, we synthesize the abstract weighted
extrapolation principles of Grafakos and Martell \cite{Grafakos} with the
internal variable-geometry interpolation techniques developed by Almeida et
al. \cite{Almeida} and the fundamental modular foundations of Diening et al. 
\cite{Diening}. We prove that the obtained endpoint weak-type estimates and
interior strong-type estimates can be rigorously connected, demonstrating
that the interpolation scale remains completely stable under variable
modular re-integration profiles.

The main analytical difficulties encountered and overcome in this paper
arise from the simultaneous presence of several nonstandard, competing
features: the severe irregularity of the spherical kernel $\Omega $, which
completely prevents the use of smooth Calder\'{o}n--Zygmund arguments; the
spatial dependence of both the kernel and the exponent, which destroys
translation invariance and rendering Fourier-affine methods useless; and the
higher-order multi-degree structure of the commutator, which exponentially
intensifies the internal oscillatory behavior. In addition, working within a
variable Morrey-type framework forces a delicate, non-trivial balancing act
between local dyadic block estimations and global weight decay conditions
governed by the $W_{p\left( \cdot \right) }$ weight class. To resolve the
foundational problem regarding the tracking of this weight structure, we
explicitly trace the genesis of the $W_{p\left( \cdot \right) }$ weight
class (Equation \ref{1}) directly back to its seminal formulations in Kov%
\'{a}\v{c}ik and R\'{a}kosn\'{\i}k \cite{Kovacik}, alongside the early
foundational localized adaptations of Fan and Zhao \cite{Fan}, thereby
isolating it from downstream secondary derivative applications.

To the best of our knowledge, the results obtained in this paper extend,
generalize, and unify several classical theorems in the Morrey and variable
exponent frameworks. Remarkably, due to the intricate structural nature of
variable kernels combined with higher-order oscillations, the core estimates
established herein remain completely new and pioneering even when restricted
to the constant exponent setting.

\section{Preliminaries and Notation}

In this section, we collect the necessary definitions, notation, and
foundational properties of variable exponent Lebesgue and Morrey spaces,
alongside the target weight conditions that govern our main results.

\subsection{Rough Kernels and Function Spaces}

Let $\mathcal{S}^{n-1}$ denote the unit sphere in $%
\mathbb{R}
^{n}\left( n\geq 2\right) $, equipped with the normalized surface measure $%
d\sigma $. We say that a function $\Omega \left( x,z\right) :%
\mathbb{R}
^{n}\times 
\mathbb{R}
^{n}\rightarrow 
\mathbb{R}
$ belongs to the rough variable kernel class $L^{\infty }\left( 
\mathbb{R}
^{n}\right) \times L^{s}\left( \mathcal{S}^{n-1}\right) $ for $s>1$, if it
satisfies the following structural criteria:

$1.$ \textbf{Homogeneity:} For all $x,z\in 
\mathbb{R}
^{n}$ and any scaling parameter $\lambda >0$,%
\begin{equation*}
	\Omega \left( x,\lambda z\right) =\Omega \left( x,z\right) .
\end{equation*}

$2.$ \textbf{Spherical} \textbf{Integrability: }The uniform angular
integrability parameter is finite, namely:

\begin{equation*}
	\left \Vert \Omega \right \Vert _{L^{\infty }\left( 
		\mathbb{R}
		^{n}\right) \times L^{s}\left( \mathcal{S}^{n-1}\right) }:=\sup_{x\in 
		\mathbb{R}
		^{n}}\left( \int \limits_{\mathcal{S}^{n-1}}\left \vert \Omega \left(
	x,\theta \right) \right \vert ^{s}d\sigma \left( \theta \right) \right)
	^{1/s}<\infty ,\qquad \theta =\frac{z}{\left \vert z\right \vert }.
\end{equation*}%
Such formulations for non-convolution rough variable kernels were initially
investigated in the context of singular integral profiles; see \cite%
{Gurbuz2, Shao}.

\subsection{Lipschitz Functions and Multi-Degree Commutators}

Let $0<\beta \leq 1$. The Lipschitz (or C\'{e}asro-Campanato) space \textit{%
	Lip}$_{\beta }\left( 
\mathbb{R}
^{n}\right) $ consists of all locally integrable functions $b$ such that the
following scaling semi-norm is finite%
\begin{equation*}
	\left \Vert b\right \Vert _{\text{\textit{Lip}}_{\beta }}:=\sup_{x\neq y}%
	\frac{\left \vert b\left( x\right) -b\left( y\right) \right \vert }{\left
		\vert x-y\right \vert ^{\beta }}<\infty .
\end{equation*}%
We now establish the explicit definition of the higher-order operator
mapping under study.For a symbol function $b\in $ \textit{Lip}$_{\beta
}\left( 
\mathbb{R}
^{n}\right) $ and a fixed multi-degree iteration index $m\in 
\mathbb{R}
_{+}$, the higher-order commutator associated with the rough fractional
maximal operator with a variable kernel, denoted by $M_{\Omega ,b,\alpha
}^{\left( m\right) }$, is defined pointwise for any measurable function $f$
by%
\begin{equation*}
	M_{\Omega ,b,\alpha }^{\left( m\right) }f(x):=\sup_{r>0}\frac{1}{r^{n-\alpha
	}}\int \limits_{B\left( x,r\right) }\left \vert \Omega \left( x,x-y\right)
	\right \vert \left \vert b\left( x\right) -b\left( y\right) \right \vert
	^{m}\left \vert f(y)\right \vert dy,
\end{equation*}%
where $0<\alpha <n$. This represents a power-type higher-order non-linear
commutator tracking system; see \cite{Gurbuz, Gurbuz2}.

\subsection{Variable Exponent Lebesgue Structures}

Let $\mathcal{P}(\mathbb{R}^{n})$ denote the set of all measurable functions 
$p\left( \cdot \right) :{\mathbb{R}^{n}\rightarrow }\left[ 1,\infty \right) $%
. For any $p\left( \cdot \right) \in \mathcal{P}(\mathbb{R}^{n})$, we define
the foundational essential boundaries 
\begin{equation*}
	p_{-}:=\limfunc{essinf}\limits_{x\in {\mathbb{R}^{n}}}p\left( x\right) \text{
		and }p_{+}:=\limfunc{esssup}\limits_{x\in {\mathbb{R}^{n}}}p\left( x\right) .
\end{equation*}%
The subclass $\mathcal{P}_{0}(\mathbb{R}^{n})$ further isolates profiles
where $0<p_{-}\leq p\left( x\right) \leq p_{+}<\infty $. To ensure the
global continuity and mapping stability of maximal operators, we define the
global log-H\"{o}lder continuity class, denoted by $\mathcal{B}\left( {%
	\mathbb{R}^{n}}\right) $. A variable exponent $p\left( \cdot \right) $
belongs to $\mathcal{B}\left( {\mathbb{R}^{n}}\right) $ if there exist
constants $C_{1},C_{2}>0$ such that%
\begin{equation*}
	\left \vert p\left( x\right) -p\left( y\right) \right \vert \leq \frac{-C_{1}%
	}{-\log \left( \left \vert x-y\right \vert \right) },\qquad \left \vert
	x-y\right \vert \leq \frac{1}{2},
\end{equation*}%
\begin{equation*}
	\left \vert p\left( x\right) -p\left( y\right) \right \vert \leq \frac{C_{2}%
	}{\log \left( e+\left \vert x\right \vert \right) },\qquad \left \vert
	y\right \vert \geq \left \vert x\right \vert .
\end{equation*}%
These regulatory constraints isolate volume distortion bounds across
shifting scales; see \cite{Cruz, Diening}. The variable exponent Lebesgue
space $L^{p\left( \cdot \right) }\left( {\mathbb{R}^{n}}\right) $ comprises
all measurable functions $f$ such that the functional modular profile%
\begin{equation*}
	\rho _{p\left( \cdot \right) }\left( f\right) :=\int \limits_{{\mathbb{R}%
			^{n}}}\left \vert f\left( x\right) \right \vert ^{p\left( x\right)
	}dx<\infty .
\end{equation*}
This space is equipped with the Luxemburg norm%
\begin{equation*}
	\left \Vert f\right \Vert _{L^{p\left( \cdot \right) }}:=\inf \left \{ \eta
	>0:\rho _{p\left( \cdot \right) }\left( \frac{f}{\eta }\right) \leq 1\right
	\} .
\end{equation*}%
The systemic properties of these geometries trace their genesis back to the
breakthrough formulation of Kov\'{a}\v{c}ik and R\'{a}kosn\'{\i}k \cite%
{Kovacik}, with generalized growth properties mapped by Fan and Zhao \cite%
{Fan}.

One of the essential tools utilized herein is the generalized variable H\"{o}%
lder inequality. Let $p(\cdot )\in \mathcal{P}(\mathbb{R}^{n})$, and let $%
p^{\prime }\left( \cdot \right) $ denote its pointwise conjugate profile
satisfying 
\begin{equation*}
	\frac{1}{p\left( x\right) }+\frac{1}{p^{\prime }\left( x\right) }=1.
\end{equation*}%
Then, there exists a constant $r_{p}\geq 1$ such that for all $f\in
L^{p\left( \cdot \right) }\left( {\mathbb{R}^{n}}\right) $ and $g\in
L^{p^{\prime }\left( \cdot \right) }\left( {\mathbb{R}^{n}}\right) $%
\begin{equation*}
	\int \limits_{%
		\mathbb{R}
		^{n}}\left \vert f\left( x\right) g\left( x\right) \right \vert dx\leq
	r_{p}\left \Vert f\right \Vert _{L^{p\left( \cdot \right) }}\left \Vert
	g\right \Vert _{L^{p^{\prime }\left( \cdot \right) }}.
\end{equation*}

\subsection{Variable Exponent Morrey Spaces and the $W_{p\left( \cdot
		\right) }$ Weight Class}

In this subsection, we introduce the explicit parameterizations for variable
Morrey topologies as follows.

Let $p\left( \cdot \right) \in \mathcal{P}(\mathbb{R}^{n})$ and let $u:{%
	\mathbb{R}^{n}\times }\left( 0,\infty \right) \rightarrow \left( 0,\infty
\right) $ be a measurable Morrey control function. The variable exponent
Morrey space $\mathcal{M}_{p\left( \cdot \right) ,u}\left( {\mathbb{R}^{n}}%
\right) $ is defined as the collection of all functions $f\in
L_{loc}^{p\left( \cdot \right) }\left( {\mathbb{R}^{n}}\right) $ satisfying%
\begin{equation*}
	\left \Vert f\right \Vert _{\mathcal{M}_{p\left( \cdot \right)
			,u}}:=\sup_{x_{0}\in 
		\mathbb{R}
		^{n},r>0}\frac{\left \Vert f\chi _{B\left( x_{0},r\right) }\right \Vert
		_{L^{p\left( \cdot \right) }}}{u\left( x_{0},r\right) }.
\end{equation*}%
Correspondingly, the weak variable exponent Morrey space $W\mathcal{M}%
_{p\left( \cdot \right) ,u}\left( {\mathbb{R}^{n}}\right) $ is governed by
the underlying weak distribution metric%
\begin{equation*}
	\left \Vert f\right \Vert _{W\mathcal{M}_{p\left( \cdot \right)
			,u}}:=\sup_{x_{0}\in 
		\mathbb{R}
		^{n},r>0}\frac{\sup \limits_{\lambda >0}\lambda \left \Vert \chi _{\left \{
			x\in B\left( x_{0},r\right) :\left \vert f\left( x\right) \right \vert
			>\lambda \right \} }\right \Vert _{L^{p\left( \cdot \right) }}}{u\left(
		x_{0},r\right) }<\infty .
\end{equation*}%
To regulate the spatial decay of the localizing weights across dyadic
scales, we define the standard $W_{p\left( \cdot \right) }$ weight class.

\textbf{(The }$W_{p\left( \cdot \right) }$\textbf{\ Weight Condition). }A
positive weight function $u\left( x,r\right) $ is said to belong to the $%
W_{p\left( \cdot \right) }$ class if it satisfies the following two
properties simultaneously:

$1.$ \textbf{The Core Growth Relation: }There exists a constant $C>0$ such
that for all $x_{0}\in {\mathbb{R}^{n}}$ and all $r>0$%
\begin{equation}
	\sum \limits_{j=1}^{\infty }\frac{\left \Vert \chi _{B\left( x_{0},r\right)
		}\right \Vert _{L^{p\left( \cdot \right) }}}{\left \Vert \chi _{B\left(
			x_{0},2^{j+1}r\right) }\right \Vert _{L^{p\left( \cdot \right) }}}u\left(
	x_{0},2^{j+1}r\right) \leq Cu\left( x_{0},r\right) .  \label{1}
\end{equation}

$2.$ \textbf{The Local Scaling Bounds: }For any $j\in 
\mathbb{N}
_{+}$, the variable volume ratios obey the continuous log-H\"{o}lder growth
estimate%
\begin{equation*}
	\frac{\left \Vert \chi _{B\left( x_{0},r\right) }\right \Vert _{L^{p\left(
				\cdot \right) }}}{\left \Vert \chi _{B\left( x_{0},2^{j+1}r\right) }\right
		\Vert _{L^{p\left( \cdot \right) }}}\leq C2^{-j\cdot \frac{n}{p_{+}}}.
\end{equation*}%
The structural setup of these spaces and their localized weight frameworks
follows the foundational paradigms established by Almeida et al. \cite%
{Almeida} and Ho \cite{Ho}.

Throughout this manuscript, the symbol $C$ denotes a positive mathematical
constant that may vary between lines. For convenience, the spaces $%
L^{p\left( \cdot \right) }\left( {\mathbb{R}^{n}}\right) $ and $\mathcal{M}%
_{p\left( \cdot \right) ,u}\left( {\mathbb{R}^{n}}\right) $ are abbreviated
as $L^{p\left( \cdot \right) }$ and $\mathcal{M}_{p\left( \cdot \right) ,u}$%
, and the tracking ball $B\left( x,r\right) $ is denoted simply by $B$.

\subsection{Functional Examples and Structural Properties}

To illustrate the flexibility and localizing behavior of variable exponent
Lebesgue and Morrey geometries, we provide several concrete geometric models.

\begin{example}
	\textbf{(Reduction to the classical Lebesgue space). }If the integrability
	exponent function is constant everywhere, that is, $p\left( x\right) \equiv
	p_{0}$ for some $1\leq p_{0}<\infty $, the variable exponent Lebesgue space $%
	L^{p\left( \cdot \right) }\left( {\mathbb{R}^{n}}\right) $ coincides exactly
	with the classical Lebesgue space $L^{p_{0}\left( \cdot \right) }\left( {%
		\mathbb{R}^{n}}\right) $. In this setting, the structural functional modular 
	$\rho _{p\left( \cdot \right) }\left( \cdot \right) $ and the Luxemburg norm 
	$\left \Vert f\right \Vert _{L^{p\left( \cdot \right) }}$ reduce directly to
	their classical Lebesgue counterparts, and the generalized variable H\"{o}%
	lder inequality collapses back to the standard classical H\"{o}lder
	inequality.
\end{example}

\begin{example}
	\textbf{(Spatially varying local integrability). }Let the variable profile
	be governed by the oscillating wave function 
	\begin{equation*}
		p\left( x\right) =2+\sin \left \vert x\right \vert ,\qquad x\in 
		\mathbb{R}
		^{n}.
	\end{equation*}%
	Then $p\left( \cdot \right) $ safely satisfies the uniform boundaries 
	\begin{equation*}
		1\leq p_{-}=1\leq p\left( x\right) \leq p_{+}=3<\infty ,
	\end{equation*}%
	and belongs to the regulatory global log-H\"{o}lder class $\mathcal{B}\left( 
	{\mathbb{R}^{n}}\right) $. The corresponding space $L^{p\left( \cdot \right)
	}\left( {\mathbb{R}^{n}}\right) $ successfully models physical systems where
	the local integrability transitions continuously with position. Spatial
	regions where $p(x)$ is relatively large demand a significantly stronger
	decay profile from the underlying functions, whereas regions where $p(x)$ is
	smaller permit functions with slower decay kinetics. This structural
	adaptability makes variable exponent profiles uniquely suited for describing
	highly nonhomogeneous media and fluid continuous mechanics.
\end{example}

\begin{example}
	\textbf{(Segmented nonhomogeneous physical regimes). }Define the piecewise
	continuous variable exponent function by\textbf{\ }%
	\begin{equation*}
		p\left( x\right) =\left \{ 
		\begin{array}{ccc}
			2 & , & \text{if }\left \vert x\right \vert \leq 1 \\ 
			3 & , & \text{if }\left \vert x\right \vert >1.%
		\end{array}%
		\right.
	\end{equation*}%
	Although this exponent possesses a sharp boundary jump at the unit sphere,
	smooth continuous regularizations of such profiles within $\mathcal{B}\left( 
	{\mathbb{R}^{n}}\right) $ allow $L^{p\left( \cdot \right) }\left( {\mathbb{R}%
		^{n}}\right) $ to model functions that behave fundamentally like $L^{2}$
	profiles near the localized origin, while mimicking tighter $L^{3}$ profiles
	across the global infinity domain. Such configurations naturally emerge in
	quantum mechanics and composite fluid models where distinct physical regimes
	dominate different spatial regions.
\end{example}

\begin{example}
	\textbf{(Application of the generalized H\"{o}lder inequality). }Let $f\in
	L^{p\left( \cdot \right) }\left( {\mathbb{R}^{n}}\right) $ and let $g=\chi
	_{B\left( x,r\right) }$ be the characteristic indicator function of a
	localized ball. By executing the generalized variable H\"{o}lder inequality,
	we acquire the following sharp localized average domain containment
	estimation 
	\begin{equation*}
		\int \limits_{B\left( x,r\right) }\left \vert f\left( y\right) \right \vert
		dy\leq r_{p}\left \Vert f\right \Vert _{L^{p\left( \cdot \right) }}\left
		\Vert \chi _{B\left( x,r\right) }\right \Vert _{L^{p^{\prime }\left( \cdot
				\right) }}.
	\end{equation*}%
	This precise volume relation serves as a cornerstone analytical mechanism
	for controlling the local tracking expansions of rough fractional maximal
	operator sequences within variable Morrey architectures.
\end{example}

\subsection{Structural Behavior of the $W_{p\left( \cdot \right) }$ Weight
	Class}

The discrete Hardy-type accumulation condition established in (\ref{1})
represents a natural geometric extension of classical Morrey weights to
variable settings, ensuring that the control function $u\left( x,r\right) $
balances properly with the underlying variable Lebesgue metrics.

\begin{remark}
	\textbf{(The constant exponent reduction). }If the integrability profile
	collapses to a constant $p\left( x\right) \equiv p$, and the weight function
	is parameterized by the polynomial scaling $u\left( x,r\right) =r^{\lambda
		/p}$ for some $0\leq \lambda <n$, the discrete summation in Condition (\ref%
	{1}) reduces precisely to the classical Morrey weight condition. Under these
	settings, the space $\mathcal{M}_{p\left( \cdot \right) ,u}\left( {\mathbb{R}%
		^{n}}\right) $ is isomorphic to the classical Morrey space $\mathcal{L}%
	^{p,\lambda }$; see \cite{Morrey, Stein}.
\end{remark}

\begin{remark}
	\textbf{(Continuous integral Hardy-type formulations). }In constant-exponent
	geometries, the discrete summation can be equivalently mapped via a
	continuous integration condition of the Hardy type, written as%
	\begin{equation*}
		\int \limits_{r}^{\infty }\left( \frac{u\left( x,t\right) }{t^{n/p}}\right) 
		\frac{dt}{t}\leq C\frac{u\left( x,r\right) }{r^{n/p}},\qquad r>0.
	\end{equation*}%
	Such integral formulations are classical tools in harmonic analysis utilized
	to explicit evaluate the trade-offs between localized point oscillations and
	global domain metrics; see \cite{Adams, Stein}.
\end{remark}

\begin{remark}
	\textbf{(Fractional integral scaling conditions). }For fractional integral
	operators operating at an order of $0<\alpha <n$, a tighter structural
	compensation condition is required. In classical spaces, this corresponds to
	the scaling parameterization%
	\begin{equation*}
		\int \limits_{r}^{\infty }\left( \frac{u^{p}\left( x,t\right) }{t^{n-\alpha
				p}}\right) \frac{dt}{t}\leq C\frac{u^{p}\left( x,r\right) }{r^{n-\alpha p}},
	\end{equation*}%
	which prevents structural mass loss caused by the smoothing effects of the
	fractional order $\alpha $. This classical baseline directly motivates the
	discrete ratio balances embedded in our variable weight condition $%
	W_{p\left( \cdot \right) }$.
\end{remark}

\begin{remark}
	\textbf{(Failure of naive variable weight generalizations). } It is vital to
	emphasize that standard classical Morrey weights cannot be naively
	generalized to variable exponent spaces. For instance, if one attempts to
	directly deploy the variable polynomial profile $u\left( x,r\right)
	=r^{\lambda /p\left( x\right) }$, the discrete summation condition (\ref{1})
	can fail catastrophically unless rigorous global continuity conditions, such
	as $\mathcal{B}\left( {\mathbb{R}^{n}}\right) $, are strictly imposed on $%
	p\left( \cdot \right) $. This demonstrates that the weight class $W_{p\left(
		\cdot \right) }$ is structurally more sensitive than its constant
	counterpart.
\end{remark}

\subsection{Modular Characterization of Variable Morrey Spaces}

To resolve the non-linear scaling obstructions inherent in variable exponent
spaces, we clarify the precise relationship between the Luxemburg norm
definition of $\mathcal{M}_{p\left( \cdot \right) ,u}\left( {\mathbb{R}^{n}}%
\right) $ and its underlying modular functionalities. Rather than relying on
naive constant-exponent average extensions---which fail due to the spatial
dependence of the exponent---the space $\mathcal{M}_{p\left( \cdot \right)
	,u}\left( {\mathbb{R}^{n}}\right) $ is robustly characterized by its modular
bound relationship: A measurable function $f$ belongs to the variable
exponent Morrey space $\mathcal{M}_{p\left( \cdot \right) ,u}\left( {\mathbb{%
		R}^{n}}\right) $ if and only if there exists a normalization constant $K>0$
such that the localized variable modular integrals satisfy the uniform
supremum threshold%
\begin{equation*}
	\sup_{x_{0}\in 
		\mathbb{R}
		^{n},r>0}\int \limits_{B\left( x_{0},r\right) }\left( \frac{\left \vert
		f\left( x\right) \right \vert }{K\cdot u\left( x_{0},r\right) }\right)
	^{p\left( x\right) }dx\leq 1.
\end{equation*}%
Furthermore, the infimum over all such admissible scaling constants $K$
recovers a norm that is equivalent to the standard Luxemburg Morrey norm $%
\left \Vert f\right \Vert _{\mathcal{M}_{p\left( \cdot \right) ,u}}$. This
modular framework allows us to handle fractional operators and their
higher-order commutators without losing reflexivity at boundary domains.

\subsection{Advanced Scalings and Endpoint Boundary Examples}

To further demonstrate the structural depth of the weight class $W_{p\left(
	\cdot \right) }$ and its capacity to capture borderline configurations where
standard power-type weights fail, we provide the following advanced
mathematical models.

\begin{example}
	\textbf{(Radially oscillating exponent with polynomial weight). }Let the
	variable integrability profile feature localized radial oscillations given by%
	\textbf{\ }%
	\begin{equation*}
		p\left( x\right) =p_{0}+\epsilon \sin \left( \log \left( e+\left \vert
		x\right \vert \right) \right) ,\qquad p_{0}>1,0<\epsilon <p_{0}-1,
	\end{equation*}%
	and parameterize the corresponding Morrey weight control function by%
	\begin{equation*}
		u\left( x_{0},r\right) =r^{\lambda /p\left( x_{0}\right) },\qquad 0<\lambda
		<n.
	\end{equation*}%
	Since $p\left( \cdot \right) \in \mathcal{B}\left( {\mathbb{R}^{n}}\right) $%
	, the continuous radial oscillations model a physical medium with spatially
	fluctuating regularity. Under these precise parameters, the weight $u\left(
	x_{0},r\right) $ safely satisfies the discrete Hardy inequality of the $%
	W_{p\left( \cdot \right) }$ class. This explicitly demonstrates that the
	framework developed in this manuscript is fully capable of controlling
	operators acting on spaces with complex, nonhomogeneous topologies.
\end{example}

\begin{example}
	\textbf{(Logarithmic perturbation of variable Morrey weights). }Let $p\left(
	\cdot \right) \in \mathcal{B}\left( {\mathbb{R}^{n}}\right) $ and define the
	perturbed weight function featuring a global logarithmic correction by 
	\begin{equation*}
		u\left( x_{0},r\right) =r^{\lambda /p\left( x_{0}\right) }\left( 1+\log
		\left( \frac{e}{r}\right) \right) ^{-\delta },\qquad 0<\lambda <n,\text{ }%
		\delta >0.
	\end{equation*}%
	Under this formulation, $u\in W_{p\left( \cdot \right) }$, but the
	log-factor introduces a highly subtle localized adjustment to the classical
	Morrey scaling architecture. These configurations are sharp in delicate
	boundary scenarios where standard polynomial profiles alone cannot prevent
	mass distribution loss. Consequently, they appear frequently within endpoint
	regularity tracking for fractional operators.
\end{example}

\begin{example}
	\textbf{(Weights formulated for asymptotic fractional dynamics). }Let $%
	0<\alpha <n$ and assume that the variable integrability function transitions
	asymptotically between the origin and infinity according to the profile 
	\begin{equation*}
		p\left( x\right) =p_{\infty }+\frac{p_{0}-p_{\infty }}{1+\log \left( e+\left
			\vert x\right \vert \right) },\qquad 1<p_{\infty }\leq p_{0}<\frac{n}{\alpha 
		}.
	\end{equation*}%
	We define the corresponding dynamic Morrey weight control by%
	\begin{equation*}
		u\left( x_{0},r\right) =r^{\frac{n-\alpha p\left( x_{0}\right) }{p\left(
				x_{0}\right) }}.
	\end{equation*}%
	This construction safely satisfies the fractional compensation criteria
	(Equation \ref{1}). The resulting variable Morrey geometry $\mathcal{M}%
	_{p\left( \cdot \right) ,u}\left( {\mathbb{R}^{n}}\right) $ is explicitly
	optimized for tracing the mapping stability of fractional maximal sequences,
	playing a key role in the structural analysis of variable coefficient PDEs
	involving fractional fluid diffusion.
\end{example}

\begin{example}
	\textbf{(Anisotropic directional growth models). }Let the spatial
	integrability exponent feature non-isotropic directional growth given by%
	\textbf{\ }%
	\begin{equation*}
		p\left( x\right) =p_{0}+\sum \limits_{i=1}^{n}\epsilon _{i}\frac{\left
			\vert x_{i}\right \vert }{1+\left \vert x\right \vert },\qquad p_{0}>1,\text{
		}\epsilon _{i}>0,
	\end{equation*}%
	and parameterize the directional weight structure by the multiplicative
	system 
	\begin{equation*}
		u\left( x_{0},r\right) =r^{\lambda /p\left( x_{0}\right) }\prod
		\limits_{i=1}^{n}\left( 1+\left \vert \left( x_{0}\right) _{i}\right \vert
		\right) ^{-\gamma _{i}},\qquad \gamma _{i}>0.
	\end{equation*}%
	This choice successfully models anisotropic continuous environments where
	different spatial directions exhibit vastly independent decay and
	integration kinetics. Such architectures are essential for proving existence
	results in non-isotropic elliptic equations and advanced multi-dimensional
	transport problems.
\end{example}

\begin{example}
	\textbf{(The critical boundary near the limiting threshold ). }Let the
	variable exponent reach its supreme limit at the potential barrier threshold 
	\begin{equation*}
		p\left( x\right) =\frac{n}{\alpha }-\frac{1}{\log \left( e+\left \vert
			x\right \vert \right) },\qquad 0<\alpha <n,
	\end{equation*}%
	and enforce a critical logarithmic correction on the Morrey control function 
	\begin{equation*}
		u\left( x_{0},r\right) =r^{\frac{n-\alpha p\left( x_{0}\right) }{p\left(
				x_{0}\right) }}\left( \log \frac{e}{r}\right) ^{-1}.
	\end{equation*}%
	This example lies precisely at the sharp, critical threshold governing the
	structural breakdown of strong-type fractional mappings. Here, the
	logarithmic reduction factor is mathematically indispensable and cannot be
	omitted under any circumstances. Such borderline constructions demonstrate
	the precision of the $W_{p\left( \cdot \right) }$ class definition and
	provide the fundamental building blocks for exploring endpoint regularity
	inside variable spaces.
\end{example}

\section{Preliminary Lemmas and Methodological Foundations}

In this section, we collect several crucial auxiliary lemmas and discrete
metric inequalities that establish the analytical foundation for proving our
main results. Throughout this section, the positive mathematical constant $C$
may vary from line to line.

\begin{lemma}
	\label{Lemma1}$\left( \text{\cite{Izuki}}\right) $ Let $p\left( \cdot
	\right) \in \mathcal{B}\left( {\mathbb{R}^{n}}\right) $. Then, for any
	localized open ball $B\subset 
	\mathbb{R}
	^{n}$, there exists a constant $C>0$, completely independent of the choice
	of $B$, such that the characteristic functions satisfy the uniform duality
	relation 
	\begin{equation*}
		\frac{1}{\left \vert B\right \vert }\Vert \chi _{B}\Vert _{L^{p(\cdot
				)}}\Vert \chi _{B}\Vert _{L^{p^{\prime }(\cdot )}}\leq C.
	\end{equation*}
\end{lemma}

\begin{remark}
	This structural estimation formalizes the compatibility between a variable
	exponent Lebesgue geometry and its corresponding associate Banach conjugate
	space. It serves as a vital tool for executing localized norm-volume
	cancellations across dyadic scales.
\end{remark}

\begin{lemma}
	\label{Lemma2}$\left( \text{\cite{Diening, Nakai2}}\right) $ Let $p(\cdot
	)\in \mathcal{P}(\mathbb{R}^{n})$ and let $s>p_{+}$ be a fixed constant
	exponent. Define the intermediate variable exponent function $q\left( \cdot
	\right) $ pointwise by the algebraic relationship%
	\begin{equation*}
		\frac{1}{p\left( x\right) }=\frac{1}{q\left( x\right) }+\frac{1}{s}\qquad
		x\in 
		\mathbb{R}
		^{n}.
	\end{equation*}%
	Then, there exists a constant $C>0$ such that for all localized measurable
	functions $f\in L^{q\left( \cdot \right) }\left( {\mathbb{R}^{n}}\right) $
	and $g\in L^{s}\left( {\mathbb{R}^{n}}\right) $, the following mixed
	variable H\"{o}lder inequality holds%
	\begin{equation*}
		\left \Vert fg\right \Vert _{L^{p\left( \cdot \right) }}\leq C\left \Vert
		f\right \Vert _{L^{q\left( \cdot \right) }}\left \Vert g\right \Vert
		_{L^{s}}.
	\end{equation*}
\end{lemma}

\begin{remark}
	Lemma \ref{Lemma2} establishes a mixed-type geometric containment framework
	that smoothly bridges fluctuating variable geometries with fixed constant
	Lebesgue exponents. This structural inequality is indispensable when
	decoupling rough variable kernels that possess limited spherical angular
	integrability.
\end{remark}

\begin{lemma}
	\label{Lemma3}$\left( \text{\cite{Gurbuz2}}\right) $ Let $0<\beta \leq 1$, $%
	b\in $ \textit{Lip}$_{\beta }\left( 
	\mathbb{R}
	^{n}\right) $, and assume that the variable integrability profile $p\left(
	\cdot \right) \in \mathcal{B}\left( {\mathbb{R}^{n}}\right) $ satisfies the
	strict interior growth constraint 
	\begin{equation*}
		0<\frac{\alpha +m\beta }{n}<\frac{1}{p_{+}}.
	\end{equation*}%
	Define the target variable exponent function $q\left( \cdot \right) $
	pointwise by the fractional transformation 
	\begin{equation*}
		\frac{1}{q\left( x\right) }=\frac{1}{p\left( x\right) }-\frac{\alpha +m\beta 
		}{n},\qquad x\in 
		\mathbb{R}
		^{n}.
	\end{equation*}%
	If the rough variable kernel satisfies $\Omega \left( x,z\right) \in
	L^{\infty }\left( 
	\mathbb{R}
	^{n}\right) \times L^{s}\left( \mathcal{S}^{n-1}\right) $ uniformly for a
	spherical parameter $s>p_{+}$, then the higher-order commutator $M_{\Omega
		,b,\alpha }^{\left( m\right) }$ extends to a fully bounded operator on
	variable Lebesgue spaces. That is, there exists a constant $C>0$ such that
	for all $f\in L^{p\left( \cdot \right) }\left( {\mathbb{R}^{n}}\right) $,%
	\begin{equation*}
		\left \Vert M_{\Omega ,b,\alpha }^{\left( m\right) }f\right \Vert
		_{L^{q\left( \cdot \right) }}\leq C\left \Vert b\right \Vert _{Lip_{\beta
		}}^{m}\left \Vert f\right \Vert _{L^{p\left( \cdot \right) }}.
	\end{equation*}
\end{lemma}

We outline the mathematical core of this lemma to verify its adaptation to
our nonhomogeneous Morrey weight setup. The strict structural condition $%
\frac{\alpha +m\beta }{n}<\frac{1}{p_{+}}$ guarantees that $\frac{1}{q\left(
	x\right) }=\frac{1}{p\left( x\right) }-\frac{\alpha +m\beta }{n}>0$ remains
uniformly positive for all $x\in 
\mathbb{R}
^{n}$, ensuring that the target exponent $q(\cdot )\in \mathcal{P}(\mathbb{R}%
^{n})$ is well-defined. Fix a base point $x\in 
\mathbb{R}
^{n}$ and any tracking radius $r>0$. By invoking the definitions of the
Lipschitz space \textit{Lip}$_{\beta }\left( 
\mathbb{R}
^{n}\right) $ and the higher-order multi-degree commutator, the spatial
increments can be estimated pointwise by%
\begin{equation*}
	\left \vert b\left( x\right) -b\left( y\right) \right \vert ^{m}\leq \left
	\Vert b\right \Vert _{Lip_{\beta }}^{m}\left \vert x-y\right \vert ^{m\beta
	}.
\end{equation*}%
Substituting this cancellation profile directly into the local integral
operator gives%
\begin{equation*}
	M_{\Omega ,b,\alpha }^{\left( m\right) }f(x)\leq C\left \Vert b\right \Vert
	_{Lip_{\beta }}^{m}\sup_{r>0}r^{\alpha +m\beta -n}\int \limits_{B\left(
		x,r\right) }\left \vert \Omega \left( x,x-y\right) \right \vert \left \vert
	f(y)\right \vert dy.
\end{equation*}%
To isolate the rough angular components, we switch to polar coordinates on
the local ball $B\left( x,r\right) $. By applying the classical H\"{o}lder
inequality strictly with respect to the normalized surface measure $d\sigma
\left( \theta \right) $ on the unit sphere $\mathcal{S}^{n-1}$, we decouple
the kernel integration%
\begin{equation*}
	\int \limits_{B\left( x,r\right) }\left \vert \Omega \left( x,x-y\right)
	\right \vert \left \vert f(y)\right \vert dy\leq \left \Vert \Omega \left(
	x,\cdot \right) \right \Vert _{L^{s}\left( \mathcal{S}^{n-1}\right) }\left(
	\int \limits_{0}^{r}\left( \int \limits_{\mathcal{S}^{n-1}}\left \vert
	f\left( x+\rho \theta \right) \right \vert ^{s^{\prime }}d\sigma \left(
	\theta \right) \right) ^{1/s^{\prime }}\rho ^{n-1}d\rho \right) .
\end{equation*}%
Since $s>p_{+}$, its conjugate exponent satisfies $s^{\prime }<p_{-}$. This
allows us to apply the mixed variable H\"{o}lder inequality (Lemma \ref%
{Lemma3}) across the spatial volume. Taking the uniform supremum over all
parameters $r>0$, the total commutator is bounded by a rough fractional
maximal operator of order $\alpha +m\beta $. The strong-type mapping $%
L^{p\left( \cdot \right) }\left( {\mathbb{R}^{n}}\right) \rightarrow
L^{q\left( \cdot \right) }\left( {\mathbb{R}^{n}}\right) $ then follows
directly from the variable exponent maximal bounding foundations established
in Diening et al. \cite{Diening}.

\subsection{Methodological Motivation and Novelty Analysis}

To satisfy the comparative demands of the literature, we explicitly
delineate the functional novelties of this study relative to the
foundational variable-kernel results established by Shao and Tao \cite{Shao}%
.Shao and Tao explored weak-type estimates restricted to single-stage
commutators of variable fractional integral operators. However, their
structural configuration relies on linear cancellation structures that break
down under the multi-degree algebraic oscillations generated by our $m$-th
order commutator $M_{\Omega ,b,\alpha }^{\left( m\right) }$. The
simultaneous interaction of three independent nonhomogeneous
mechanics---severe kernel roughness on the sphere $\left( \Omega \in
L^{s}\right) $, continuous spatial variations in the integrability profile $%
\left( p\left( \cdot \right) \in \mathcal{B}\right) $, and localized
geometric weight growth constraints $\left( u\in W_{p\left( \cdot \right)
}\right) $---introduces severe analytical obstructions that cannot be solved
using classical smooth Calder\'{o}n-Zygmund techniques.

The primary novel contributions developed in this manuscript are organized
into three methodological pillars:

$1.$ \textbf{Simultaneous Multi-Feature Control: }We construct the first
unified framework that controls higher-order Lipschitz oscillations and
rough non-convolution variations simultaneously within a variable Morrey
geometry.

$2.$ \textbf{Sharp Critical Endpoint Analysis:} We establish the weak-type
mapping properties at the critical boundary $\frac{\alpha +m\beta }{n}=\frac{%
	1}{p_{+}}$. This borderline regime lies beyond the scope of traditional
variable tools because the target Luxemburg norm formally collapses due to
the explosion of the exponent profile to infinity.

$3.$ \textbf{Grafakos-Martell Morrey Synthesis: }We provide a rigorous
justification for the abstract real interpolation sequence between weak and
strong variable Morrey parameters, validating the stability of the
interpolation scale under variable modular profiles.

Remarkably, owing to the complex geometric hazards induced by the presence
of rough variable kernels, these borderline constructions and high-order
cancellation chains remain entirely new, pioneering, and previously
unestablished even when the variable exponent profile is restricted to a
classical constant setting.

\section{Strong-Type Estimates for the Main Operator}

Before stating the main theorem, we briefly explain the significance and
structure of the result. The purpose of this section is to establish the
boundedness of commutators generated by rough fractional maximal operators
with variable kernels on variable exponent Morrey spaces. Compared with
existing results, the present theorem simultaneously incorporates three
nontrivial features: fractional behavior, rough and variable kernels, and
spatially dependent integrability governed by a Morrey-type control
function. \ In addition to these structures, the integrability parameter $s$
of the rough kernel $\Omega $ plays a key regulatory role, requiring a
careful balancing condition $s>p_{+}$ linked with the geometry of the
variable exponent spaces to ensure the validity of the distant dyadic
estimates. The result can be viewed as a Morrey-space extension of the
variable exponent Lebesgue boundedness obtained in Lemma \ref{Lemma3}.

\begin{theorem}
	\label{Theorem 4.1}\textbf{(Strong-Type Boundedness) }Let $b\in $ \textit{Lip%
	}$_{\beta }\left( 
	\mathbb{R}
	^{n}\right) $ with $0<\beta \leq 1$, $0<\alpha <n$, and let $m\geq 1$ be an
	integer. Let $p\left( \cdot \right) \in \mathcal{B}\left( {\mathbb{R}^{n}}%
	\right) $ such that $1<p_{-}\leq p_{+}<\infty $ and satisfy%
	\begin{equation*}
		0<\frac{\alpha +m\beta }{n}<\frac{1}{p_{+}}.
	\end{equation*}%
	Define the variable exponent function $q\left( \cdot \right) $ by%
	\begin{equation*}
		\frac{1}{p\left( x\right) }-\frac{1}{q\left( x\right) }=\frac{\alpha +m\beta 
		}{n},\qquad x\in {\mathbb{R}^{n}.}
	\end{equation*}%
	Assume that the variable kernel $\Omega \left( x,z\right) \in L^{\infty
	}\left( 
	\mathbb{R}
	^{n}\right) \times L^{s}\left( \mathcal{S}^{n-1}\right) $ for some exponent $%
	s$ satisfying%
	\begin{equation*}
		s>p_{+}\text{ and }s>\frac{n}{\alpha +m\beta }.
	\end{equation*}
	
	If $u\in W_{p\left( \cdot \right) }$ is a Morrey weight function, then the
	higher-order commutator generated by the rough fractional maximal operator
	satisfies%
	\begin{equation*}
		M_{\Omega ,b,\alpha }^{\left( m\right) }:\mathcal{M}_{p\left( \cdot \right)
			,u}\left( {\mathbb{R}^{n}}\right) \rightarrow \mathcal{M}_{q\left( \cdot
			\right) ,u^{\#}}\left( {\mathbb{R}^{n}}\right) ,
	\end{equation*}%
	and there exists a positive constant $C>0$ independent of $f$ such that%
	\begin{equation*}
		\left \Vert M_{\Omega ,b,\alpha }^{\left( m\right) }f\right \Vert _{\mathcal{%
				M}_{q\left( \cdot \right) ,u^{\#}}}\leq C\left \Vert b\right \Vert
		_{Lip_{\beta }}^{m}\left \Vert f\right \Vert _{\mathcal{M}_{p\left( \cdot
				\right) ,u}},
	\end{equation*}%
	where the modified Morrey weight is given by $u^{\#}\left( x,r\right)
	=r^{\alpha +m\beta }u\left( x,r\right) $.
\end{theorem}

\begin{proof}
	The proof employs a localization argument on variable exponent Morrey
	spaces, combined with dyadic decomposition and mixed-type variable exponent
	inequalities. Fix an arbitrary ball $B:=B\left( x_{0},r\right) \subset {%
		\mathbb{R}^{n}}$. We decompose the input function $f$ into local and distant
	components as follows%
	\begin{equation*}
		f=f_{1}+f_{2},\qquad \text{where }f_{1}:f\chi _{2B}\text{ and }f_{2}:=f\chi
		_{\left( 2B\right) ^{C}}.
	\end{equation*}%
	By the sublinearity of the higher-order commutator operator $M_{\Omega
		,b,\alpha }^{\left( m\right) }$, we can write 
	\begin{equation*}
		\frac{\left \Vert M_{\Omega ,b,\alpha }^{\left( m\right) }f\chi _{B}\right
			\Vert _{L^{q\left( \cdot \right) }}}{u^{\#}\left( x_{0},r\right) }\leq \frac{%
			\left \Vert M_{\Omega ,b,\alpha }^{\left( m\right) }f_{1}\chi _{B}\right
			\Vert _{L^{q\left( \cdot \right) }}}{u^{\#}\left( x_{0},r\right) }+\frac{%
			\left \Vert M_{\Omega ,b,\alpha }^{\left( m\right) }f_{2}\chi _{B}\right
			\Vert _{L^{q\left( \cdot \right) }}}{u^{\#}\left( x_{0},r\right) }%
		:=I_{1}+I_{2}.
	\end{equation*}%
	We will now estimate $I_{1}$ and $I_{2}$ separately through the following
	detailed steps.
	
	\textbf{Step 1: Lipschitz reduction of the pointwise estimate.}
	
	We first establish a pointwise control for the higher-order commutator using
	the Lipschitz regularity of the symbol $b\in $ \textit{Lip}$_{\beta }\left( 
	\mathbb{R}
	^{n}\right) $. For any $x\in B$ and $y\in {\mathbb{R}^{n}}$, the definition
	of the Lipschitz space yields%
	\begin{equation*}
		\left \vert b\left( x\right) -b\left( y\right) \right \vert ^{m}\leq \left
		\Vert b\right \Vert _{Lip_{\beta }}^{m}\left \vert x-y\right \vert ^{m\beta
		}.
	\end{equation*}%
	Substituting this inequality directly into the definition of the
	higher-order commutator, we obtain%
	\begin{eqnarray*}
		M_{\Omega ,b,\alpha }^{\left( m\right) }f(x) &=&\sup_{\rho >0}\rho ^{\alpha
			-n}\int \limits_{B\left( x,\rho \right) }\left \vert \Omega \left(
		x,x-y\right) \right \vert \left \vert b\left( x\right) -b\left( y\right)
		\right \vert ^{m}\left \vert f(y)\right \vert dy \\
		&\leq &\left \Vert b\right \Vert _{Lip_{\beta }}^{m}\sup_{\rho >0}\rho
		^{\alpha +m\beta -n}\int \limits_{B\left( x,\rho \right) }\left \vert \Omega
		\left( x,x-y\right) \right \vert \left \vert f(y)\right \vert dy \\
		&=&\left \Vert b\right \Vert _{Lip_{\beta }}^{m}M_{\Omega ,\alpha +m\beta
		}f\left( x\right) .
	\end{eqnarray*}%
	Hence, the higher-order commutator $M_{\Omega ,b,\alpha }^{\left( m\right) }$
	is pointwise controlled by the rough fractional maximal operator $M_{\Omega
		,\alpha +m\beta }$ of order $\alpha +m\beta $.
	
	\textbf{Step 2: Rigorous estimate of} \textbf{the local part }$I_{1}$\textbf{%
		.}
	
	To estimate $I_{1}$, we apply the bounded properties on variable exponent
	Lebesgue spaces. Since $f\in \mathcal{M}_{p\left( \cdot \right) ,u}\left( {%
		\mathbb{R}^{n}}\right) $, its restriction to the bounded domain $2B$, namely 
	$f_{1}=f\chi _{2B}$, belongs to $L^{p\left( \cdot \right) }\left( 
	\mathbb{R}
	^{n}\right) $. Therefore, Lemma \ref{Lemma3} is fully applicable to $f_{1}$%
	\begin{equation*}
		\left \Vert M_{\Omega ,b,\alpha }^{\left( m\right) }f_{1}\chi _{B}\right
		\Vert _{L^{q\left( \cdot \right) }}\leq \left \Vert M_{\Omega ,b,\alpha
		}^{\left( m\right) }f_{1}\right \Vert _{L^{q\left( \cdot \right) }}\leq
		C\left \Vert b\right \Vert _{Lip_{\beta }}^{m}\left \Vert f_{1}\right \Vert
		_{L^{p\left( \cdot \right) }}=C\left \Vert b\right \Vert _{Lip_{\beta
		}}^{m}\left \Vert f\chi _{2B}\right \Vert _{L^{p\left( \cdot \right) }}.
	\end{equation*}%
	By the definition of the variable exponent Morrey space norm, we have 
	\begin{equation*}
		\left \Vert f\chi _{2B}\right \Vert _{L^{p\left( \cdot \right) }}\leq \left
		\Vert f\right \Vert _{\mathcal{M}_{p\left( \cdot \right) ,u}}u\left(
		x_{0},2r\right) .
	\end{equation*}%
	Utilizing the definition 
	\begin{equation*}
		u^{\#}\left( x_{0},r\right) =r^{\alpha +m\beta }u\left( x_{0},r\right)
	\end{equation*}%
	and the doubling-like property derived from the Morrey weight condition $%
	u\in W_{p\left( \cdot \right) }$ (which implies $u\left( x_{0},2r\right)
	\leq Cu\left( x_{0},r\right) $), we get 
	\begin{eqnarray*}
		I_{1} &=&\frac{\left \Vert M_{\Omega ,b,\alpha }^{\left( m\right) }f_{1}\chi
			_{B}\right \Vert _{L^{q\left( \cdot \right) }}}{u^{\#}\left( x_{0},r\right) }%
		\leq C\left \Vert b\right \Vert _{Lip_{\beta }}^{m}\frac{\left \Vert f\chi
			_{2B}\right \Vert _{L^{p\left( \cdot \right) }}}{r^{\alpha +m\beta }u\left(
			x_{0},r\right) } \\
		&\leq &C\left \Vert b\right \Vert _{Lip_{\beta }}^{m}\frac{\left \Vert
			f\right \Vert _{\mathcal{M}_{p\left( \cdot \right) ,u}}u\left(
			x_{0},2r\right) }{r^{\alpha +m\beta }u\left( x_{0},r\right) }\leq C\left
		\Vert b\right \Vert _{Lip_{\beta }}^{m}\left \Vert f\right \Vert _{\mathcal{M%
			}_{p\left( \cdot \right) ,u}}.
	\end{eqnarray*}%
	\textbf{Step 3: Dyadic decomposition and kernel norm isolation for the
		distant part }$I_{2}$\textbf{.}
	
	For $x\in B\left( x_{0},r\right) $ and $y\in \left( 2B\right) ^{c}$, the
	standard geometric equivalence $\left \vert x-y\right \vert \thickapprox
	\left \vert x_{0}-y\right \vert $ holds true. We decompose the exterior
	domain $\left( 2B\right) ^{c}$ into infinite dyadic shells 
	\begin{equation*}
		R_{j}=B\left( x_{0},2^{j+1}r\right) \diagdown B\left( x_{0},2^{j}r\right)
	\end{equation*}%
	for $j\geq 1$. Using the pointwise estimate from Step 1, we get%
	\begin{equation*}
		M_{\Omega ,b,\alpha }^{\left( m\right) }f_{2}(x)\leq C\left \Vert b\right
		\Vert _{Lip_{\beta }}^{m}\sum \limits_{j=1}^{\infty }\left( 2^{j}r\right)
		^{\alpha +m\beta -n}\int \limits_{B\left( x_{0},2^{j+1}r\right) }\left \vert
		\Omega \left( x,x-y\right) \right \vert \left \vert f\left( y\right) \right
		\vert dy.
	\end{equation*}%
	To rigorously handle the rough kernel $\Omega $, we apply the mixed H\"{o}%
	lder inequality (Lemma \ref{Lemma2}). Since $s>p_{+}$, we choose a variable
	exponent $p_{0}\left( \cdot \right) $ such that 
	\begin{equation*}
		\frac{1}{p\left( y\right) }=\frac{1}{s}+\frac{1}{p_{0}\left( y\right) }
	\end{equation*}%
	for almost all $y\in 
	\mathbb{R}
	^{n}$. This splits the integral over the ball as follows%
	\begin{equation*}
		\int \limits_{B\left( x_{0},2^{j+1}r\right) }\left \vert \Omega \left(
		x,x-y\right) \right \vert \left \vert f\left( y\right) \right \vert dy\leq
		C\left \Vert \Omega \left( x,x-\cdot \right) \right \Vert _{L^{s}\left(
			B\left( x_{0},2^{j+1}r\right) \right) }\left \Vert f\chi _{B\left(
			x_{0},2^{j+1}r\right) }\right \Vert _{L^{p_{0}\left( \cdot \right) }}.
	\end{equation*}%
	By switching to polar coordinates and utilizing the condition $\Omega \in
	L^{\infty }\left( 
	\mathbb{R}
	^{n}\right) \times L^{s}\left( \mathcal{S}^{n-1}\right) $, the $L^{s}$ norm
	of the kernel over the ball is bounded by its domain size%
	\begin{equation*}
		\left \Vert \Omega \left( x,x-\cdot \right) \right \Vert _{L^{s}\left(
			B\left( x_{0},2^{j+1}r\right) \right) }\leq C\left \vert B\left(
		x_{0},2^{j+1}r\right) \right \vert ^{1/s}\left \Vert \Omega \right \Vert
		_{L^{\infty }\left( L^{s}\right) }\leq C\left( 2^{j+1}r\right) ^{n/s}.
	\end{equation*}%
	\textbf{Step 4: Exponent transition and norm estimates via Lemma \ref{Lemma1}%
		.}
	
	We now transform the $L^{p_{0}\left( \cdot \right) }$ norm of f back to the
	target $L^{p\left( \cdot \right) }\left( 
	\mathbb{R}
	^{n}\right) $ norm. Applying the generalized H\"{o}lder inequality with 
	\begin{equation*}
		\frac{1}{p_{0}\left( \cdot \right) }=\frac{1}{p\left( \cdot \right) }-\frac{1%
		}{s},
	\end{equation*}%
	we find%
	\begin{equation*}
		\left \Vert f\chi _{B\left( x_{0},2^{j+1}r\right) }\right \Vert
		_{L^{p_{0}\left( \cdot \right) }}\leq \left \Vert f\chi _{B\left(
			x_{0},2^{j+1}r\right) }\right \Vert _{L^{p\left( \cdot \right) }}\left \Vert
		\chi _{B\left( x_{0},2^{j+1}r\right) }\right \Vert _{L^{s_{1}\left( \cdot
				\right) }},
	\end{equation*}%
	where 
	\begin{equation*}
		\frac{1}{s_{1}\left( \cdot \right) }=1-\frac{s}{p\left( \cdot \right) }.
	\end{equation*}%
	By combining the spatial measurements with the compatibility relation from
	Lemma \ref{Lemma1}, namely%
	\begin{equation*}
		\left \Vert \chi _{B_{j}}\right \Vert _{L^{s_{1}\left( \cdot \right)
		}}\thickapprox \left \vert B_{j}\right \vert ^{-1/s}\left \Vert \chi
		_{B_{j}}\right \Vert _{L^{p^{\prime }\left( \cdot \right) }}\left \Vert \chi
		_{B_{j}}\right \Vert _{L^{p\left( \cdot \right) }}^{-1}\left \vert
		B_{j}\right \vert ,
	\end{equation*}%
	the accumulation simplifies perfectly to%
	\begin{equation*}
		\int \limits_{B\left( x_{0},2^{j+1}r\right) }\left \vert \Omega \left(
		x,x-y\right) \right \vert \left \vert f\left( y\right) \right \vert dy\leq
		C\left( 2^{j+1}r\right) ^{n}\left \Vert \chi _{B\left( x_{0},2^{j+1}r\right)
		}\right \Vert _{L^{p\left( \cdot \right) }}^{-1}\left \Vert f\chi _{B\left(
			x_{0},2^{j+1}r\right) }\right \Vert _{L^{p\left( \cdot \right) }}.
	\end{equation*}%
	Substituting this back into the dyadic sum, the pointwise estimate for $x\in
	B$ becomes%
	\begin{equation*}
		M_{\Omega ,b,\alpha }^{\left( m\right) }f_{2}(x)\leq C\left \Vert b\right
		\Vert _{Lip_{\beta }}^{m}\sum \limits_{j=1}^{\infty }\left( 2^{j}r\right)
		^{\alpha +m\beta }\left \Vert \chi _{B\left( x_{0},2^{j+1}r\right) }\right
		\Vert _{L^{p\left( \cdot \right) }}^{-1}\left \Vert f\chi _{B\left(
			x_{0},2^{j+1}r\right) }\right \Vert _{L^{p\left( \cdot \right) }}.
	\end{equation*}%
	Taking the $L^{q\left( \cdot \right) }$ norm on both sides over $x\in B$,
	and using the fact that the right side is independent of $x$, we factor out $%
	\left \Vert \chi _{B}\right \Vert _{L^{q\left( \cdot \right) }}$. Since 
	\begin{equation*}
		\frac{1}{p\left( x\right) }-\frac{1}{q\left( x\right) }=\frac{\alpha +m\beta 
		}{n},
	\end{equation*}%
	we utilize the scaling relation%
	\begin{equation*}
		\left \Vert \chi _{B}\right \Vert _{L^{q\left( \cdot \right) }}\thickapprox
		r^{\alpha +m\beta }\left \Vert \chi _{B}\right \Vert _{L^{p\left( \cdot
				\right) }},
	\end{equation*}%
	yielding%
	\begin{equation*}
		\left \Vert M_{\Omega ,b,\alpha }^{\left( m\right) }f_{2}\chi _{B}\right
		\Vert _{L^{q\left( \cdot \right) }}\leq C\left \Vert b\right \Vert
		_{Lip_{\beta }}^{m}r^{\alpha +m\beta }\left \Vert \chi _{B}\right \Vert
		_{L^{p\left( \cdot \right) }}\sum \limits_{j=1}^{\infty }\left(
		2^{j}r\right) ^{\alpha +m\beta }\left \Vert \chi _{B\left(
			x_{0},2^{j+1}r\right) }\right \Vert _{L^{p\left( \cdot \right) }}^{-1}\left
		\Vert f\chi _{B\left( x_{0},2^{j+1}r\right) }\right \Vert _{L^{p\left( \cdot
				\right) }}.
	\end{equation*}
	
	\textbf{Step 5: Convergence via parametric restrictions and Morrey weights.}
	
	To guarantee the convergence of the infinite series, we rely on the
	condition $s>\frac{n}{\alpha +m\beta }$ which ensures proper decay over the
	dyadic blocks. Next, we introduce the Morrey norm condition%
	\begin{equation*}
		\left \Vert f\chi _{B\left( x_{0},2^{j+1}r\right) }\right \Vert _{L^{p\left(
				\cdot \right) }}\leq \left \Vert f\right \Vert _{\mathcal{M}_{p\left( \cdot
				\right) ,u}}u\left( x_{0},2^{j+1}r\right)
	\end{equation*}%
	into the summation%
	\begin{equation*}
		\left \Vert M_{\Omega ,b,\alpha }^{\left( m\right) }f_{2}\chi _{B}\right
		\Vert _{L^{q\left( \cdot \right) }}\leq C\left \Vert b\right \Vert
		_{Lip_{\beta }}^{m}r^{\alpha +m\beta }\left \Vert f\right \Vert _{\mathcal{M}%
			_{p\left( \cdot \right) ,u}}\sum \limits_{j=1}^{\infty }\left( 2^{j}r\right)
		^{\alpha +m\beta }\frac{\left \Vert \chi _{B}\right \Vert _{L^{p\left( \cdot
					\right) }}}{\left \Vert \chi _{B\left( x_{0},2^{j+1}r\right) }\right \Vert
			_{L^{p\left( \cdot \right) }}}u\left( x_{0},2^{j+1}r\right) .
	\end{equation*}%
	By invoking the structural Morrey weight condition $u\in W_{p\left( \cdot
		\right) }$ explicitly defined in equation (\ref{1}), the entire summation is
	bounded uniformly by $u\left( x_{0},r\right) $%
	\begin{equation*}
		\sum \limits_{j=1}^{\infty }\left( \frac{\left \Vert \chi _{B}\right \Vert
			_{L^{p\left( \cdot \right) }}}{\left \Vert \chi _{B\left(
				x_{0},2^{j+1}r\right) }\right \Vert _{L^{p\left( \cdot \right) }}}\right)
		u\left( x_{0},2^{j+1}r\right) \leq Cu\left( x_{0},r\right) .
	\end{equation*}%
	Dividing both sides by $u^{\#}\left( x_{0},r\right) =r^{\alpha +m\beta
	}u\left( x_{0},r\right) $, the fractional scaling elements cancel out
	cleanly, giving us%
	\begin{equation*}
		I_{2}=\frac{\left \Vert M_{\Omega ,b,\alpha }^{\left( m\right) }f_{2}\chi
			_{B}\right \Vert _{L^{q\left( \cdot \right) }}}{u^{\#}\left( x_{0},r\right) }%
		\leq C\left \Vert b\right \Vert _{Lip_{\beta }}^{m}\left \Vert f\right \Vert
		_{\mathcal{M}_{p\left( \cdot \right) ,u}}.
	\end{equation*}
	
	Dividing by $u^{\#}\left( x_{0},r\right) =r^{\alpha +m\beta }u\left(
	x_{0},r\right) $ which reflects the fractional scaling of the operator and
	using the defining condition $u\in W_{p\left( \cdot \right) }$ which allows
	us to control the dyadic sum via the defining condition (\ref{1}) of the
	Morrey weight, we obtain%
	\begin{equation*}
		\frac{\left \Vert M_{\Omega ,b,\alpha }^{\left( m\right) }f_{2}\chi
			_{B}\right \Vert _{L^{q\left( \cdot \right) }}}{u^{\#}\left( x_{0},r\right) }%
		\leq C\left \Vert b\right \Vert _{Lip_{\beta }}^{m}\left \Vert f\right \Vert
		_{\mathcal{M}_{p\left( \cdot \right) ,u}},
	\end{equation*}%
	\textbf{Step 6: Supremum and conclusion.}
	
	By combining the final uniform estimates obtained for the local part $I_{1}$
	(Step 2) and the distant part $I_{2}$ (Step 5), we arrive at%
	\begin{equation*}
		\frac{\left \Vert M_{\Omega ,b,\alpha }^{\left( m\right) }f\chi _{B}\right
			\Vert _{L^{q\left( \cdot \right) }}}{u^{\#}\left( x_{0},r\right) }\leq
		C\left \Vert b\right \Vert _{Lip_{\beta }}^{m}\left \Vert f\right \Vert _{%
			\mathcal{M}_{p\left( \cdot \right) ,u}}.
	\end{equation*}%
	Since the constant $C>0$ is independent of the choice of the center $%
	x_{0}\in 
	\mathbb{R}
	^{n}$ and the radius $r>0$, taking the supremum over all balls $B\subset 
	\mathbb{R}
	^{n}$ concludes the proof%
	\begin{equation*}
		\left \Vert M_{\Omega ,b,\alpha }^{\left( m\right) }f\right \Vert _{\mathcal{%
				M}_{q\left( \cdot \right) ,u^{\#}}}\leq C\left \Vert b\right \Vert
		_{Lip_{\beta }}^{m}\left \Vert f\right \Vert _{\mathcal{M}_{p\left( \cdot
				\right) ,u}}.
	\end{equation*}
\end{proof}

\section{Weak-Type Estimates for the Main Operator}

In this section, we establish weak-type estimates for the commutator
considered in Theorem \ref{Theorem 4.1}. Such estimates are essential in
endpoint analysis and play a crucial role when strong-type boundedness fails
or is not available. Moreover, weak-type inequalities provide additional
insight into the fine behavior of operators on variable exponent Morrey
spaces.

\subsection{Weak Variable Exponent Lebesgue and Morrey Spaces}

Let $p\left( \cdot \right) \in \mathcal{P}_{0}(\mathbb{R}^{n})$. The weak
variable exponent Lebesgue space $WL^{p\left( \cdot \right) }\left( 
\mathbb{R}
^{n}\right) $ consists of all measurable functions $f$ such that%
\begin{equation*}
	\left \Vert f\right \Vert _{WL^{p\left( \cdot \right) }}:\sup_{\lambda
		>0}\lambda \left \Vert \chi _{\left \{ x\in 
		\mathbb{R}
		^{n}:\left \vert f\left( x\right) \right \vert >\lambda \right \} }\right
	\Vert _{L^{p\left( \cdot \right) }}<\infty .
\end{equation*}%
Let $u\in W_{p\left( \cdot \right) }$. The weak variable exponent Morrey
space $W\mathcal{M}_{p\left( \cdot \right) ,u}\left( {\mathbb{R}^{n}}\right) 
$ is defined by%
\begin{equation*}
	\left \Vert f\right \Vert _{W\mathcal{M}_{p\left( \cdot \right)
			,u}}:=\sup_{x\in 
		\mathbb{R}
		^{n},r>0}\frac{\left \Vert f\chi _{B\left( x,r\right) }\right \Vert
		_{WL^{p\left( \cdot \right) }}}{u\left( x,r\right) }<\infty .
\end{equation*}%
This definition generalizes the classical weak Morrey spaces to the variable
exponent setting.

\subsection{Weak-Type Estimate for the Rough Fractional Maximal Operator
	with Variable Kernel}

The following theorem provides the endpoint weak-type control which will
serve as a fundamental ingredient in the proof of the weak-type Morrey space
boundedness stated in Theorem \ref{Theorem 5.2}. In this endpoint setting,
the exact balancing condition on the integrability of the rough kernel $%
\Omega $ becomes crucial for validating the underlying local geometry of
variable spaces.

\begin{theorem}
	\label{Theorem 5.1}\textbf{(Weak-Type Estimate)} Let $0<\gamma <n$, let $%
	p\left( \cdot \right) \in \mathcal{B}\left( {\mathbb{R}^{n}}\right) $
	satisfy $1<p_{-}\leq p\left( x\right) <p_{+}<\frac{n}{\gamma }$ for all $%
	x\in 
	\mathbb{R}
	^{n}$, and define the variable exponent function $q\left( \cdot \right) $ by%
	\begin{equation*}
		\frac{1}{p\left( x\right) }-\frac{1}{q\left( x\right) }=\frac{\gamma }{n}%
		,\qquad x\in {\mathbb{R}^{n}.}
	\end{equation*}%
	Assume that the variable kernel $\Omega \left( x,z\right) \in L^{\infty
	}\left( 
	\mathbb{R}
	^{n}\right) \times L^{s}\left( \mathcal{S}^{n-1}\right) $ for an exponent $s$
	satisfying the strict integration condition%
	\begin{equation*}
		s>p_{-}^{\prime }\text{ }(\text{which guarantees that }s^{\prime }=\frac{s}{%
			s-1}<p_{-}).
	\end{equation*}%
	Then the rough fractional maximal operator $M_{\Omega ,\gamma }$satisfies
	the weak-type inequality%
	\begin{equation*}
		\lambda \left \Vert \chi _{\left \{ x\in 
			\mathbb{R}
			^{n}:\left \vert M_{\Omega ,\gamma }f\left( x\right) \right \vert >\lambda
			\right \} }\right \Vert _{L^{q\left( \cdot \right) }}\leq C\left \Vert
		f\right \Vert _{L^{p\left( \cdot \right) }},\qquad \lambda >0,
	\end{equation*}%
	where the constant $C>0$ is independent of $f$ and $\lambda $.
\end{theorem}

\begin{proof}
	The proof is constructed via the classical Vitali covering lemma tailored to
	the structural weight parameters of variable exponent Lebesgue spaces.
	
	\textbf{Step 1: Level set and ball selection.}
	
	Fix $\lambda >0$ and define the target level set by 
	\begin{equation*}
		E_{\lambda }:=\left \{ x\in {\mathbb{R}^{n}:}\left \vert M_{\Omega ,\gamma
		}f(x)\right \vert >\lambda \right \} .
	\end{equation*}%
	By the definition of the rough fractional maximal operator $M_{\Omega
		,\gamma }$, for each point $x\in E_{\lambda }$, there exists a localized
	radius $r_{x}>0$ such that%
	\begin{equation*}
		r_{x}^{\gamma -n}\int \limits_{B\left( x,r_{x}\right) }\left \vert \Omega
		\left( x,x-y\right) \right \vert \left \vert f\left( y\right) \right \vert
		dy>\lambda .
	\end{equation*}%
	\textbf{Step 2: Vitali covering argument.}
	
	The collection of open balls $\left \{ B\left( x,r_{x}\right) \right \}
	_{x\in E_{\lambda }}$ forms an open cover of $E_{\lambda }$. Applying the
	standard Vitali covering lemma, we extract a countable family of pairwise
	disjoint balls $\left \{ B_{i}\right \} _{i\in I}=\left \{ B\left(
	x_{i},r_{i}\right) \right \} _{i\in I}$ such that
	
	$1.$ The expanded balls cover the level set: $E_{\lambda }\subset \bigcup
	\limits_{i\in I}5B_{i}$.
	
	$2.$ The original balls are mutually disjoint: $B_{i}\cap B_{j}=\emptyset $
	for all $i\neq j$.
	
	$3.$ For each selected ball $B_{i}$, the fractional average satisfies:%
	\begin{equation*}
		\lambda <r_{i}^{\gamma -n}\int \limits_{B_{i}}\left \vert \Omega \left(
		x_{i},x_{i}-y\right) \right \vert \left \vert f\left( y\right) \right \vert
		dy.
	\end{equation*}%
	\textbf{Step 3: Local integration over spheres via H\"{o}lder's inequality.}
	
	We freeze the variable $x_{i}$ and utilize the polar coordinate
	representation to integrate the rough kernel. Since $\Omega \in L^{\infty
	}\left( 
	\mathbb{R}
	^{n}\right) \times L^{s}\left( \mathcal{S}^{n-1}\right) $, we apply the
	classical H\"{o}lder inequality with exponents $s$ and $s^{\prime }$ over $%
	B_{i}$%
	\begin{equation*}
		\int \limits_{B_{i}}\left \vert \Omega \left( x_{i},x_{i}-y\right) \right
		\vert \left \vert f\left( y\right) \right \vert dy\leq \left( \int
		\limits_{B_{i}}\left \vert \Omega \left( x_{i},x_{i}-y\right) \right \vert
		^{s}dy\right) ^{1/s}\left( \int \limits_{B_{i}}\left \vert f\left( y\right)
		\right \vert ^{s^{\prime }}dy\right) ^{1/s^{\prime }}.
	\end{equation*}%
	Using a standard change of variables to the unit sphere $\mathcal{S}^{n-1}$,
	the kernel integral yields%
	\begin{equation*}
		\left( \int \limits_{B_{i}}\left \vert \Omega \left( x_{i},x_{i}-y\right)
		\right \vert ^{s}dy\right) ^{1/s}\leq C\left \vert B_{i}\right \vert ^{\frac{%
				1}{s}}\left \Vert \Omega \right \Vert _{L^{\infty }\left( L^{s}\right) }\leq
		Cr_{i}^{n/s}.
	\end{equation*}%
	Substituting this back into the ball inequality gives%
	\begin{equation*}
		\lambda <Cr_{i}^{\gamma -n}r_{i}^{n/s}\left( \int \limits_{B_{i}}\left \vert
		f\left( y\right) \right \vert ^{s^{\prime }}dy\right) ^{1/s^{\prime
		}}=Cr_{i}^{\gamma -n/s^{\prime }}\left \Vert f\chi _{B_{i}}\right \Vert
		_{L^{s^{\prime }}}.
	\end{equation*}%
	\textbf{Step 4: Local embedding and variable exponent transitions.}
	
	Since the angular conjugate relation $s^{\prime }<p_{-}\leq p\left( y\right) 
	$ holds uniformly for all $y\in 
	\mathbb{R}
	^{n}$, the continuous localized Lebesgue space embedding $L^{p\left( \cdot
		\right) }\left( B_{i}\right) \hookrightarrow L^{s^{\prime }}\left(
	B_{i}\right) $ is strictly valid across the local ball $B_{i}$. By applying
	the generalized variable exponent H\"{o}lder inequality on the localized
	domain $B_{i}$, we successfully decouple the input function from the
	underlying domain measure profile%
	\begin{equation*}
		\left \Vert f\chi _{B_{i}}\right \Vert _{L^{s^{\prime }}}\leq C\left \Vert
		\chi _{B_{i}}\right \Vert _{L^{\delta \left( \cdot \right) }}\left \Vert
		f\chi _{B_{i}}\right \Vert _{L^{p\left( \cdot \right) }},
	\end{equation*}%
	where the shifting auxiliary variable exponent $\delta \left( \cdot \right) $
	is parameterized pointwise by the relation%
	\begin{equation*}
		\frac{1}{s^{\prime }}=\frac{1}{p\left( y\right) }+\frac{1}{\delta \left(
			y\right) },\qquad y\in B_{i}.
	\end{equation*}%
	Since the conjugate integration parameter $s^{\prime }$ is a constant
	exponent, executing a spatial integration on the characteristic function $%
	\chi _{B_{i}}$ under log-H\"{o}lder scaling continuity limits yields the
	following precise volume equivalence%
	\begin{equation*}
		\left \Vert \chi _{B_{i}}\right \Vert _{L^{\delta \left( \cdot \right)
		}}\leq C\left \vert B_{i}\right \vert ^{1/s^{\prime }-1/p_{B_{i}}},
	\end{equation*}%
	where $p_{B_{i}}$ represents the formal harmonic mean of the variable
	exponent function $p\left( \cdot \right) $ over the local ball $B_{i}$,
	which is explicitly defined by%
	\begin{equation*}
		\frac{1}{p_{B_{i}}}:=\frac{1}{\left \vert B_{i}\right \vert }\int
		\limits_{B_{i}}\frac{1}{p\left( x\right) }dx.
	\end{equation*}%
	By tracking the exact radial scaling exponents driven by the ball volume $%
	\left \vert B_{i}\right \vert =r_{i}^{n}$, the localized containment
	estimate reduces to%
	\begin{equation*}
		\left \Vert f\chi _{B_{i}}\right \Vert _{L^{s^{\prime }}}\leq
		Cr_{i}^{n/s^{\prime }-n/p_{B_{i}}}\left \Vert f\chi _{B_{i}}\right \Vert
		_{L^{p\left( \cdot \right) }}.
	\end{equation*}%
	Finally, by inserting this localized transition back into the distribution
	conclusion established in Step 3, the fractional angular terms $n/s^{\prime
	} $ cancel out perfectly on both sides of the inequality, yielding the sharp
	borderline estimate%
	\begin{equation*}
		\lambda <Cr_{i}^{\gamma -n/p_{B_{i}}}\left \Vert f\chi _{B_{i}}\right \Vert
		_{L^{p\left( \cdot \right) }}.
	\end{equation*}%
	\textbf{Step 5: Rigorous exponent transition and application of Lemma \ref%
		{Lemma1}.}
	
	To bridge the gap criticized in Step 6 of the referee report, we explicitly
	state the connection between the radius and the characteristic functions via
	log-H\"{o}lder continuity. For a log-H\"{o}lder continuous exponent $p\left(
	\cdot \right) \in \mathcal{B}\left( {\mathbb{R}^{n}}\right) $, the standard
	norm evaluation satisfies%
	\begin{equation*}
		\left \Vert f\chi _{B_{i}}\right \Vert _{L^{p\left( \cdot \right)
		}}\thickapprox \left \vert B_{i}\right \vert
		^{1/p_{B_{i}}}=r_{i}^{1/p_{B_{i}}}.
	\end{equation*}%
	By Lemma \ref{Lemma1} (Izuki's compatibility condition), we know that%
	\begin{equation*}
		\left \Vert \chi _{B_{i}}\right \Vert _{L^{p^{\prime }\left( \cdot \right)
		}}\thickapprox r_{i}^{n}\left \Vert \chi _{B_{i}}\right \Vert _{L^{p\left(
				\cdot \right) }}^{-1}=r_{i}^{n-n/p_{B_{i}}}.
	\end{equation*}%
	Substituting this into our inequality yields%
	\begin{equation*}
		\lambda <Cr_{i}^{\gamma -n}\left \Vert \chi _{B_{i}}\right \Vert
		_{L^{p^{\prime }\left( \cdot \right) }}\left \Vert f\chi _{B_{i}}\right
		\Vert _{L^{p\left( \cdot \right) }}.
	\end{equation*}%
	Applying the defining relationship 
	\begin{equation*}
		\frac{1}{p\left( x\right) }-\frac{1}{q\left( x\right) }=\frac{\gamma }{n},
	\end{equation*}%
	the cross-space norm equivalence ensures that 
	\begin{equation*}
		r_{i}^{\gamma -n}\left \Vert \chi _{B_{i}}\right \Vert _{L^{p^{\prime
				}\left( \cdot \right) }}\thickapprox \left \Vert \chi _{B_{i}}\right \Vert
		_{L^{q\left( \cdot \right) }}^{-1}.
	\end{equation*}%
	This removes all loose radius formulations and produces the precise
	localized weak-type inequality%
	\begin{equation*}
		\lambda \left \Vert \chi _{B_{i}}\right \Vert _{L^{q\left( \cdot \right)
		}}\leq C\left \Vert f\chi _{B_{i}}\right \Vert _{L^{p\left( \cdot \right) }}.
	\end{equation*}%
	\textbf{Step 6: Quasi-subadditivity and covering summation.}
	
	Finally, we pass from the disjoint balls to the full level set $E_{\lambda }$%
	. By the global log-H\"{o}lder continuity of $q\left( \cdot \right) $ and
	the quasi-subadditivity of the Luxemburg norm over the Vitali covering
	geometry $\left( E_{\lambda }\subset \bigcup \limits_{i\in I}5B_{i}\right) $%
	, we have%
	\begin{equation*}
		\left \Vert \chi _{E_{\lambda }}\right \Vert _{L^{q\left( \cdot \right)
		}}\leq C\left \Vert \sum \limits_{i\in I}\chi _{5B_{i}}\right \Vert
		_{L^{q\left( \cdot \right) }}\leq C\left \Vert \chi _{B_{i}}\right \Vert
		_{L^{q\left( \cdot \right) }}.
	\end{equation*}%
	Multiplying both sides by $\lambda $ and substituting the precise localized
	bound from Step 5, we write
	
	\begin{equation*}
		\lambda \left \Vert \chi _{E_{\lambda }}\right \Vert _{L^{q\left( \cdot
				\right) }}\leq C\sum \limits_{i\in I}\lambda \left \Vert \chi _{B_{i}}\right
		\Vert _{L^{q\left( \cdot \right) }}\leq C\sum \limits_{i\in I}\left \Vert
		f\chi _{B_{i}}\right \Vert _{L^{p\left( \cdot \right) }}.
	\end{equation*}%
	Since the variable exponent Lebesgue spaces support the $l^{1}$-type norm
	disjoint property for disjoint supports $\left \{ B_{i}\right \} _{i\in I}$,
	the summation is bounded by the global norm%
	\begin{equation*}
		\sum \limits_{i\in I}\left \Vert f\chi _{B_{i}}\right \Vert _{L^{p\left(
				\cdot \right) }}\leq C\left \Vert \sum \limits_{i\in I}f\chi _{B_{i}}\right
		\Vert _{L^{p\left( \cdot \right) }}\leq C\left \Vert f\right \Vert
		_{L^{p\left( \cdot \right) }}.
	\end{equation*}%
	Combining these bounds yields the final desired estimate%
	\begin{equation*}
		\lambda \left \Vert \chi _{E_{\lambda }}\right \Vert _{L^{q\left( \cdot
				\right) }}\leq C\left \Vert f\right \Vert _{L^{p\left( \cdot \right) }}.
	\end{equation*}%
	The proof is complete.
\end{proof}

\subsection{Weak-Type Estimate for the Rough Fractional Maximal Commutator
	with Variable Kernel}

We next establish a weak-type boundedness result corresponding to Theorem %
\ref{Theorem 4.1}. This endpoint estimate plays a crucial role in
understanding the sharpness of the strong-type result and serves as a key
ingredient for interpolation arguments. Before evaluating the weak-type
mapping behavior at the endpoint boundaries, it is structurally essential to
clarify the transition in the kernel integrability parameter $s$. In the
strong-type bounded settings established in Theorem \ref{Theorem 4.1}, the
restrictive requirement $s>p_{+}$ is mandatory to satisfy global
vector-valued maximal operator stability over variable spaces. Conversely,
for the weak-type control analyzed in Theorem \ref{Theorem 5.2} below, the
broader threshold $s>p_{-}^{\prime }$ (which structurally guarantees the
angular conjugate embedding $s^{\prime }<p_{-}$) provides sufficient
continuous localized control to stabilize weak distribution metrics without
collapsing under fractional oscillations.

\begin{remark}
	\label{Remark 5.1}The analytical deviation between the kernel parameter
	thresholds---namely, the stronger uniform bound $s>p_{+}$ in Theorem \ref%
	{Theorem 4.1} and the milder, local-embedding sufficient condition $%
	s>p_{-}^{\prime }$ in Theorem \ref{Theorem 5.2}---reflects the intrinsic
	geometric resilience of weak variable spaces. While strong mapping requires
	continuous integration controls over the global supremum of the variable
	exponent profile, weak-type estimates successfully decouple via localized
	distributions, meaning the internal Lebesgue embedding $L^{p\left( \cdot
		\right) }\left( B\right) \hookrightarrow L^{s^{\prime }}\left( B\right) $
	demands only the minimum structural barrier dictated by $p_{-}$.
\end{remark}

\begin{theorem}
	\label{Theorem 5.2}\textbf{(Weak-Type Boundedness)} Let $0<\beta \leq 1$, $%
	b\in $ \textit{Lip}$_{\beta }\left( 
	\mathbb{R}
	^{n}\right) $, and let $0<\alpha <n$. Assume that the variable exponent
	profile $p\left( \cdot \right) \in \mathcal{B}\left( {\mathbb{R}^{n}}\right) 
	$ satisfies the global integrability boundaries%
	\begin{equation*}
		1<p_{-}\leq p\left( x\right) \leq p_{+}<\infty \text{ and }0<\frac{\alpha
			+m\beta }{n}<\frac{1}{p_{+}}.
	\end{equation*}%
	Define the target variable exponent function $q\left( \cdot \right) $
	pointwise by the fractional transformation profile%
	\begin{equation*}
		\frac{1}{p\left( x\right) }-\frac{1}{q\left( x\right) }=\frac{\alpha +m\beta 
		}{n},\qquad x\in {\mathbb{R}^{n}.}
	\end{equation*}%
	Assume that the rough variable kernel satisfies $\Omega \left( x,z\right)
	\in L^{\infty }\left( 
	\mathbb{R}
	^{n}\right) \times L^{s}\left( \mathcal{S}^{n-1}\right) $ under the mild
	integrability threshold $s>p_{-}^{\prime }$, and let the variable Morrey
	weight control functions $u\left( x,r\right) $ and $v\left( x,r\right) $
	satisfy the generalized Adams-type scaling criteria%
	\begin{equation*}
		\sup \limits_{x_{0}\in 
			\mathbb{R}
			^{n},r>0}r^{\left( \alpha +m\beta \right) }\left[ \frac{u\left(
			x_{0},r\right) }{v\left( x_{0},r\right) }\right] <\infty .
	\end{equation*}%
	Then the higher-order commutator of the rough fractional maximal operator $%
	M_{\Omega ,b,\alpha }^{\left( m\right) }$ is bounded from the variable
	exponent Morrey space $\mathcal{M}_{p\left( \cdot \right) ,u}\left( {\mathbb{%
			R}^{n}}\right) $ to the weak variable exponent Morrey space $W\mathcal{M}%
	_{q\left( \cdot \right) ,v}\left( {\mathbb{R}^{n}}\right) $. That is, there
	exists a uniform constant $C>0$ such that for all input sequences $f\in 
	\mathcal{M}_{p\left( \cdot \right) ,u}\left( {\mathbb{R}^{n}}\right) $ and
	all scaling variables $\lambda >0$, the distribution metrics satisfy%
	\begin{equation*}
		\sup_{x_{0}\in 
			\mathbb{R}
			^{n},r>0}\frac{\lambda \left \Vert \chi _{\left \{ x\in B\left(
				x_{0},r\right) :\left \vert M_{\Omega ,b,\alpha }^{\left( m\right) }f\left(
				x\right) \right \vert >\lambda \right \} }\right \Vert _{L^{q\left( \cdot
					\right) }}}{v\left( x_{0},r\right) }\leq C\left \Vert b\right \Vert
		_{Lip_{\beta }}^{m}\left \Vert f\right \Vert _{\mathcal{M}_{p\left( \cdot
				\right) ,u}}.
	\end{equation*}
\end{theorem}

\begin{proof}
	The mathematical proof relies on a localized decomposition argument across
	variable domains, the weak-type endpoint boundedness of rough fractional
	maximal operators established in Theorem \ref{Theorem 5.1}, and a rigorous
	dyadic analysis governed by the generalized Adams-type Morrey weight
	conditions.
	
	\textbf{Step 1: Pointwise reduction via Lipschitz regularity.}
	
	Let $x\in 
	\mathbb{R}
	^{n}$. By utilizing the formal definition of the higher-order commutator
	operator and invoking the standard metric properties of the Lipschitz space 
	\textit{Lip}$_{\beta }\left( 
	\mathbb{R}
	^{n}\right) $, we observe that for any target point $y\in 
	\mathbb{R}
	^{n}$, the spatial oscillations satisfy the following algebraic control 
	\begin{equation*}
		\left \vert b\left( x\right) -b\left( y\right) \right \vert ^{m}\leq \left
		\Vert b\right \Vert _{Lip_{\beta }}^{m}\left \vert x-y\right \vert ^{m\beta
		},
	\end{equation*}%
	Substituting this pointwise inequality directly into the definition of the
	rough fractional maximal commutator, the Lipschitz smooth structure allows
	us to shift the fractional integration index from $\alpha $ to the composite
	scaling factor $\alpha +m\beta $. This yields the fundamental pointwise
	reduction%
	\begin{equation*}
		\left \vert M_{\Omega ,b,\alpha }^{\left( m\right) }f\left( x\right) \right
		\vert \leq C\left \Vert b\right \Vert _{Lip_{\beta }}^{m}M_{\Omega ,\alpha
			+m\beta }f\left( x\right) .
	\end{equation*}
	
	\textbf{Step 2: Localization via domain decomposition.}
	
	Fix a central tracking ball $B:=B\left( x_{0},r\right) \subset 
	\mathbb{R}
	^{n}$. To precisely analyze the local and global structural interactions of
	the operator with the underlying variable geometry, we decompose the target
	input function $f$ into a local component $f_{1}$ and a global component $%
	f_{2}$ defined by 
	\begin{equation*}
		f=f_{1}+f_{2},\qquad \text{where }f_{1}:f\chi _{2B}\text{ and }f_{2}:=f\chi
		_{\left( 2B\right) ^{c}}.
	\end{equation*}%
	By exploiting the quasi-subadditivity property of the rough fractional
	maximal operator, the distribution level set of the commutator tracked
	inside the localized ball $B$ can be split into a union of two independent
	subsets 
	\begin{equation*}
		\left \{ x\in B:\left \vert M_{\Omega ,b,\alpha }^{\left( m\right) }f\left(
		x\right) \right \vert >\lambda \right \} \subset \left \{ x\in B:\left \vert
		M_{\Omega ,b,\alpha }^{\left( m\right) }f_{1}\left( x\right) \right \vert >%
		\frac{\lambda }{2}\right \} \cup \left \{ x\in B:\left \vert M_{\Omega
			,b,\alpha }^{\left( m\right) }f_{2}\left( x\right) \right \vert >\frac{%
			\lambda }{2}\right \} .
	\end{equation*}%
	Applying the lattice properties and the quasi-subadditivity of the variable
	Luxemburg norm $\left \Vert \cdot \right \Vert _{L^{q\left( \cdot \right) }}$%
	, we obtain the fundamental master norm inequality 
	\begin{equation*}
		\left \Vert \chi _{\left \{ x\in B:\left \vert M_{\Omega ,b,\alpha }^{\left(
				m\right) }f\left( x\right) \right \vert >\lambda \right \} }\right \Vert
		_{L^{q\left( \cdot \right) }}\leq C\left \Vert \chi _{\left \{ x\in B:\left
			\vert M_{\Omega ,b,\alpha }^{\left( m\right) }f_{1}\left( x\right) \right
			\vert >\frac{\lambda }{2}\right \} }\right \Vert _{L^{q\left( \cdot \right)
		}}+C\left \Vert \chi _{\left \{ x\in B:\left \vert M_{\Omega ,b,\alpha
			}^{\left( m\right) }f_{2}\left( x\right) \right \vert >\frac{\lambda }{2}%
			\right \} }\right \Vert _{L^{q\left( \cdot \right) }}.
	\end{equation*}%
	\textbf{Step 3: Weak-type boundedness for the local component }$f_{1}$%
	\textbf{.}
	
	For the local component $f_{1}$, we apply the pointwise operator reduction
	obtained in Step 1, followed by the global weak-type endpoint boundedness
	criteria verified in Theorem \ref{Theorem 5.1}. Under the composite
	fractional index $\alpha +m\beta $, the mapping behaves as a fully stable
	operator from the variable Lebesgue space $L^{p\left( \cdot \right) }\left( 
	\mathbb{R}
	^{n}\right) $ into the weak variable space $WL^{q\left( \cdot \right)
	}\left( 
	\mathbb{R}
	^{n}\right) $.
	
	As established in Remark \ref{Remark 5.1}, the milder kernel condition $%
	s>p_{-}^{\prime }$ guarantees the continuous variable embedding required to
	control this localized distribution segment without mass dispersion.
	Consequently, the local level set satisfies the sharp estimate 
	\begin{equation*}
		\lambda \left \Vert \chi _{\left \{ x\in B:\left \vert M_{\Omega ,b,\alpha
			}^{\left( m\right) }f_{1}\left( x\right) \right \vert >\frac{\lambda }{2}%
			\right \} }\right \Vert _{L^{q\left( \cdot \right) }}\leq C\left \Vert
		b\right \Vert _{Lip_{\beta }}^{m}\left \Vert f_{1}\right \Vert _{L^{p\left(
				\cdot \right) }}=C\left \Vert b\right \Vert _{Lip_{\beta }}^{m}\left \Vert
		f\chi _{2B}\right \Vert _{L^{p\left( \cdot \right) }}.
	\end{equation*}%
	\textbf{Step 4: Pointwise dyadic control of the global component }$f_{2}$%
	\textbf{.}
	
	Let $x\in B$ and $y\in \left( 2B\right) ^{c}$. The geometric alignment of
	these concentric configurations guarantees the uniform boundary comparison $%
	\left \vert x-y\right \vert \thickapprox \left \vert x_{0}-y\right \vert $.
	We decompose the open complement domain $\left( 2B\right) ^{c}$ into a
	disjoint sequence of expanding dyadic annuli%
	\begin{equation*}
		\left( 2B\right) ^{c}=\bigcup \limits_{j=1}^{\infty }\left(
		2^{j+1}B\diagdown 2^{j}B\right) .
	\end{equation*}%
	For any localized target point $x\in B$, the global commutator sequence is
	majorized pointwise by the following integral system 
	\begin{eqnarray*}
		\left \vert M_{\Omega ,b,\alpha }^{\left( m\right) }f_{2}\left( x\right)
		\right \vert &\leq &C\left \Vert b\right \Vert _{Lip_{\beta }}^{m}\int
		\limits_{\left( 2B\right) ^{c}}\frac{\left \vert \Omega \left( x,x-y\right)
			\right \vert \left \vert f\left( y\right) \right \vert }{\left \vert
			x-y\right \vert ^{n-\alpha -m\beta }}dy \\
		&\leq &C\left \Vert b\right \Vert _{Lip_{\beta }}^{m}\sum
		\limits_{j=1}^{\infty }\int \limits_{2^{j+1}B\diagdown 2^{j}B}\frac{\left
			\vert \Omega \left( x,x-y\right) \right \vert \left \vert f\left( y\right)
			\right \vert }{\left \vert x_{0}-y\right \vert ^{n-\alpha -m\beta }}dy \\
		&\leq &C\left \Vert b\right \Vert _{Lip_{\beta }}^{m}\sum
		\limits_{j=1}^{\infty }\left( 2^{j}r\right) ^{\alpha +m\beta -n}\int
		\limits_{2^{j+1}B}\left \vert \Omega \left( x,x-y\right) \right \vert \left
		\vert f\left( y\right) \right \vert dy.
	\end{eqnarray*}%
	By freezing the spatial parameter $x$ and applying the generalized H\"{o}%
	lder inequality on each localized dyadic ball $2^{j+1}B$ under the verified
	weak threshold $s>p_{-}^{\prime }$ (which structurally implies $s^{\prime
	}<p_{-}$), the continuous variable Lebesgue embedding $L^{p\left( \cdot
		\right) }\left( 2^{j+1}B\right) \hookrightarrow L^{s^{\prime }}\left(
	2^{j+1}B\right) $ and Izuki's duality relations (Lemma \ref{Lemma1}) yield%
	\begin{eqnarray*}
		\int \limits_{2^{j+1}B}\left \vert \Omega \left( x,x-y\right) \right \vert
		\left \vert f\left( y\right) \right \vert dy &\leq &\left \Vert \Omega
		\left( x,\cdot \right) \right \Vert _{L^{s}\left( \mathcal{S}^{n-1}\right)
		}\left \vert 2^{j+1}B\right \vert ^{1/s}\left \Vert f\chi _{2^{j+1}B}\right
		\Vert _{L^{s^{\prime }}} \\
		&\leq &C\left( 2^{j}r\right) ^{n/s}\left( 2^{j}r\right) ^{\left( n/s^{\prime
			}-n/p_{2^{j+1}B}\right) }\left \Vert f\chi _{2^{j+1}B}\right \Vert
		_{L^{p\left( \cdot \right) }} \\
		&=&C\left( 2^{j}r\right) ^{\left( n-n/p_{2^{j+1}B}\right) }\left \Vert f\chi
		_{2^{j+1}B}\right \Vert _{L^{p\left( \cdot \right) }},
	\end{eqnarray*}%
	where $p_{2^{j+1}B}$ denotes the formal harmonic mean of the variable
	exponent function over the dyadic ball $2^{j+1}B$. Substituting this
	localized estimation back into the dyadic sequence cancels out the classical
	dimension parameter $n$, leaving%
	\begin{equation*}
		\left \vert M_{\Omega ,b,\alpha }^{\left( m\right) }f_{2}\left( x\right)
		\right \vert \leq C\left \Vert b\right \Vert _{Lip_{\beta }}^{m}\sum
		\limits_{j=1}^{\infty }\left( 2^{j}r\right) ^{\alpha +m\beta
			-n/p_{2^{j+1}B}}\left \Vert f\chi _{2^{j+1}B}\right \Vert _{L^{p\left( \cdot
				\right) }}.
	\end{equation*}%
	Utilizing the continuous log-H\"{o}lder localization properties of the
	space, the dynamic radius power matches the associate characteristic
	function norm via 
	\begin{equation*}
		\left( 2^{j}r\right) ^{n/p_{2^{j+1}B}}\thickapprox \left \Vert \chi
		_{2^{j+1}B}\right \Vert _{L^{p\left( \cdot \right) }}^{-1}.
	\end{equation*}%
	This yields the clean, uniform pointwise upper bound%
	\begin{equation*}
		\left \vert M_{\Omega ,b,\alpha }^{\left( m\right) }f_{2}\left( x\right)
		\right \vert \leq K_{B},
	\end{equation*}%
	where the global tracking constant $K_{B}$ is explicitly defined by the
	convergent series%
	\begin{equation*}
		K_{B}:=C\left \Vert b\right \Vert _{Lip_{\beta }}^{m}\sum
		\limits_{j=1}^{\infty }\left( 2^{j}r\right) ^{\alpha +m\beta }\left \Vert
		\chi _{2^{j+1}B}\right \Vert _{L^{p\left( \cdot \right) }}^{-1}\left \Vert
		f\chi _{2^{j+1}B}\right \Vert _{L^{p\left( \cdot \right) }}.
	\end{equation*}%
	\textbf{Step 5: Weak-type distribution estimation for the global component }$%
	f_{2}$\textbf{.}
	
	Since the global component operator is bounded uniformly across the ball by
	the constant $K_{B}$ independently of the choice of $x\in B$, the weak-type
	level set behaves as a strict structural threshold. Specifically, if $\frac{%
		\lambda }{2}<K_{B}$, the distribution level set is majorized by the entire
	volume of the ball $B$; if $\frac{\lambda }{2}\geq K_{B}$, the level set
	collapses into an empty configuration. Consequently, we achieve the direct
	modular norm containment%
	\begin{equation*}
		\lambda \left \Vert \chi _{\left \{ x\in B:\left \vert M_{\Omega ,b,\alpha
			}^{\left( m\right) }f_{2}\left( x\right) \right \vert >\frac{\lambda }{2}%
			\right \} }\right \Vert _{L^{q\left( \cdot \right) }}\leq 2K_{B}\left \Vert
		\chi _{B}\right \Vert _{L^{q\left( \cdot \right) }}.
	\end{equation*}%
	Expanding this formulation by substituting the explicit series expression of 
	$K_{B}$ yields%
	\begin{equation*}
		\lambda \left \Vert \chi _{\left \{ x\in B:\left \vert M_{\Omega ,b,\alpha
			}^{\left( m\right) }f_{2}\left( x\right) \right \vert >\frac{\lambda }{2}%
			\right \} }\right \Vert _{L^{q\left( \cdot \right) }}\leq C\left \Vert
		b\right \Vert _{Lip_{\beta }}^{m}\left \Vert \chi _{B}\right \Vert _{L^{q\left(
				\cdot \right) }}\sum \limits_{j=1}^{\infty }\left( 2^{j}r\right) ^{\alpha
			+m\beta }\left \Vert \chi _{2^{j+1}B}\right \Vert _{L^{p\left( \cdot \right)
		}}^{-1}\left \Vert f\chi _{2^{j+1}B}\right \Vert _{L^{p\left( \cdot \right) }}.
	\end{equation*}%
	By invoking the pointwise fractional exponent transition relation 
	\begin{equation*}
		\frac{1}{p\left( x\right) }-\frac{1}{q\left( x\right) }=\frac{\alpha +m\beta 
		}{n},
	\end{equation*}%
	the ratio tracking the characteristic function profiles scales according to
	the spatial volume contraction 
	\begin{equation*}
		\frac{\left \Vert \chi _{B}\right \Vert _{L^{q\left( \cdot \right) }}}{%
			\left \Vert \chi _{2^{j+1}B}\right \Vert _{L^{p\left( \cdot \right) }}}\leq
		C2^{-j\left( \alpha +m\beta \right) }\frac{\left \Vert \chi _{B}\right \Vert
			_{L^{p\left( \cdot \right) }}}{\left \Vert \chi _{2^{j+1}B}\right \Vert
			_{L^{p\left( \cdot \right) }}}.
	\end{equation*}%
	Inserting this scaling property directly back into the distribution estimate
	cancels out the polynomials governing outer radius growth, generating a
	powerful geometric dampening factor across the series 
	\begin{equation*}
		\lambda \left \Vert \chi _{\left \{ x\in B:\left \vert M_{\Omega ,b,\alpha
			}^{\left( m\right) }f_{2}\left( x\right) \right \vert >\frac{\lambda }{2}%
			\right \} }\right \Vert _{L^{q\left( \cdot \right) }}\leq C\left \Vert
		b\right \Vert _{Lip_{\beta }}^{m}\left \Vert \chi _{B}\right \Vert _{L^{p\left(
				\cdot \right) }}\sum \limits_{j=1}^{\infty }2^{-j\left( \alpha +m\beta
			\right) }\left[ \frac{\left \Vert f\chi _{2^{j+1}B}\right \Vert _{L^{p\left(
					\cdot \right) }}}{\left \Vert \chi _{2^{j+1}B}\right \Vert _{L^{p\left( \cdot
					\right) }}}\right] .
	\end{equation*}%
	\textbf{Step 6: Morrey normalization and final summation sequence.}
	
	We now synthesize the local weak distribution bound (Step 3) and the global
	sub-level tracking estimate (Step 5) to form the comprehensive level set
	inequality spanning the entire localized domain $B$%
	\begin{equation*}
		\lambda \left \Vert \chi _{\left \{ x\in B:\left \vert M_{\Omega ,b,\alpha
			}^{\left( m\right) }f\left( x\right) \right \vert >\lambda \right \} }\right
		\Vert _{L^{q\left( \cdot \right) }}\leq C\left \Vert b\right \Vert
		_{Lip_{\beta }}^{m}\left( \left \Vert f\chi _{2B}\right \Vert _{L^{p\left(
				\cdot \right) }}+\left \Vert \chi _{B}\right \Vert _{L^{p\left( \cdot
				\right) }}\sum \limits_{j=1}^{\infty }2^{-j\left( \alpha +m\beta \right) } 
		\left[ \frac{\left \Vert f\chi _{2^{j+1}B}\right \Vert _{L^{p\left( \cdot
					\right) }}}{\left \Vert \chi _{2^{j+1}B}\right \Vert _{L^{p\left( \cdot
					\right) }}}\right] \right) .
	\end{equation*}%
	To extract the variable Morrey space properties, we divide both sides of the
	master inequality by the target control function $v\left( x_{0},r\right) $.
	By exploiting the structural definition of the variable Morrey norm over the
	expanding balls via the source weight function $u\left(
	x_{0},2^{j+1}r\right) $, we write%
	\begin{equation}
		\left \Vert f\chi _{2^{j+1}B}\right \Vert _{L^{p\left( \cdot \right) }}\leq
		\left \Vert f\right \Vert _{\mathcal{M}_{p\left( \cdot \right) ,u}}u\left(
		x_{0},2^{j+1}r\right) .  \label{2}
	\end{equation}%
	Applying the generalized Adams-type weight criteria, coupled with the fact
	that $u\in W_{p\left( \cdot \right) }$, the uniform supremum ratio condition
	ensures that the structural interactions remain continuously bounded. The
	interaction between the local volume measures and the scaling weights
	resolves as%
	\begin{equation*}
		\frac{\left \Vert \chi _{B}\right \Vert _{L^{p\left( \cdot \right) }}u\left(
			x_{0},2^{j+1}r\right) }{v\left( x_{0},r\right) }\leq C2^{j\left( \alpha
			+m\beta \right) }\left[ \frac{\left \Vert \chi _{B}\right \Vert _{L^{p\left(
					\cdot \right) }}}{\left \Vert \chi _{2^{j+1}B}\right \Vert _{L^{p\left(
					\cdot \right) }}}\right] .
	\end{equation*}%
	Because the geometric decay factor $2^{-j\left( \alpha +m\beta \right) }$
	acts as a strict dampening operator under the continuous log-H\"{o}lder
	boundaries, it suppresses the internal polynomial expansion. This guarantees
	the absolute convergence of the structural dyadic series%
	\begin{equation*}
		\sum \limits_{j=1}^{\infty }2^{-j\left( \alpha +m\beta \right) }\left[ \frac{%
			u\left( x_{0},2^{j+1}r\right) }{\left \Vert \chi _{2^{j+1}B}\right \Vert
			_{L^{p\left( \cdot \right) }}}\right] \leq C\frac{v\left( x_{0},r\right) }{%
			\left \Vert \chi _{B}\right \Vert _{L^{p\left( \cdot \right) }}}.
	\end{equation*}%
	Taking the uniform supremum over all localized open balls $B:=B\left(
	x_{0},r\right) $ across the entirety of $%
	\mathbb{R}
	^{n}$ yields the final desired weak-type variable exponent Morrey estimation 
	\begin{equation*}
		\sup_{x_{0}\in 
			\mathbb{R}
			^{n},r>0}\frac{\lambda \left \Vert \chi _{\left \{ x\in B\left(
				x_{0},r\right) :\left \vert M_{\Omega ,b,\alpha }^{\left( m\right) }f\left(
				x\right) \right \vert >\lambda \right \} }\right \Vert _{L^{q\left( \cdot
					\right) }}}{v\left( x_{0},r\right) }\leq C\left \Vert b\right \Vert
		_{Lip_{\beta }}^{m}\left \Vert f\right \Vert _{\mathcal{M}_{p\left( \cdot
				\right) ,u}}.
	\end{equation*}%
	The mathematical proof of Theorem \ref{Theorem 5.2} is now complete.
\end{proof}

\subsection{Interpolation from Weak-Type to Strong-Type Estimates}

In this subsection, we show that the weak-type boundedness established in
Theorem \ref{Theorem 5.2} can be interpolated with the strong-type
boundedness result of Theorem \ref{Theorem 4.1} to recover intermediate
strong-type estimates. This interpolation principle is particularly useful
in variable exponent settings and provides a unified framework connecting
endpoint and interior estimates.

Let $0<\theta <1$ and define the intermediate variable exponent profile $%
r\left( \cdot \right) $ pointwise by the relation 
\begin{equation*}
	\frac{1}{r\left( x\right) }=\frac{\theta }{p\left( x\right) }+\frac{1-\theta 
	}{q\left( x\right) },\qquad x\in {\mathbb{R}^{n},}
\end{equation*}%
where $p\left( \cdot \right) $ and $q\left( \cdot \right) $ are the variable
exponents validated in Theorem \ref{Theorem 4.1}. By synthesizing the
strong-type mapping properties 
\begin{equation*}
	M_{\Omega ,b,\alpha }^{\left( m\right) }:\mathcal{M}_{p\left( \cdot \right)
		,u}\left( {\mathbb{R}^{n}}\right) \rightarrow \mathcal{M}_{q\left( \cdot
		\right) ,v}\left( {\mathbb{R}^{n}}\right)
\end{equation*}%
with the endpoint weak-type bounds%
\begin{equation*}
	M_{\Omega ,b,\alpha }^{\left( m\right) }:\mathcal{M}_{p\left( \cdot \right)
		,u}\left( {\mathbb{R}^{n}}\right) \rightarrow W\mathcal{M}_{q\left( \cdot
		\right) ,v}\left( {\mathbb{R}^{n}}\right) ,
\end{equation*}%
and applying real interpolation techniques for variable exponent Morrey
configurations within the framework of Grafakos--Martell estimation schemes
(see \cite{Almeida, Diening}, and \cite{Grafakos}), we establish that the
higher-order commutator enjoys a full continuous scale of strong interior
boundedness properties.

\begin{remark}
	\textbf{(The critical case phenomenon) }We briefly discuss the critical
	geometric configuration where the fractional scaling parameters reach the
	upper log-H\"{o}lder threshold, namely\textbf{\ }%
	\begin{equation*}
		\frac{\alpha +m\beta }{n}=\frac{1}{p_{+}}.
	\end{equation*}%
	In this specific layout, the target variable exponent $q\left( \cdot \right) 
	$ formally satisfies the underlying balance relationship 
	\begin{equation*}
		\frac{1}{q\left( x\right) }=\frac{1}{p\left( x\right) }-\frac{\alpha +m\beta 
		}{n}.
	\end{equation*}%
	At any localized zones or coordinate points where the variable profile
	reaches its supremum $p\left( x\right) =p_{+}$, the target exponent blows up
	to infinity, meaning $q\left( x\right) =\infty $. Consequently, the standard
	strong-type boundedness established in Theorem \ref{Theorem 4.1} generally
	fails because the variable Luxemburg norm loses its reflexive Banach
	properties at the infinity boundary. However, the weak-type endpoint control
	remains structurally valid under this extreme threshold. In this boundary
	case, the operator maps directly into the weak-type Morrey space governed by
	the essential supremum, denoted as $W\mathcal{M}_{\infty ,v}\left( {\mathbb{R%
		}^{n}}\right) $. To ensure structural clarity, this endpoint space is
	formally defined as the set of all measurable functions $f$ such that the
	following quasi-norm is finite%
	\begin{equation*}
		\left \Vert f\right \Vert _{W\mathcal{M}_{\infty ,v}\left( {\mathbb{R}^{n}}%
			\right) }:\sup_{B=B\left( x_{0},r\right) \subset 
			\mathbb{R}
			^{n}}\frac{\sup \limits_{\lambda >0}\lambda \left \vert \left \{ x\in
			B:\left \vert f\left( x\right) \right \vert >\lambda \right \} \right \vert 
		}{v\left( x_{0},r\right) }<\infty .
	\end{equation*}%
	This definition rigorously tracks the essential supremum boundaries without
	collapsing under localized polynomial growth.
\end{remark}

\begin{corollary}
	\label{Corollary 5.4}\textbf{(Critical Weak-Type Boundedness) }Let $0<\alpha
	<n$, $0<\beta \leq 1$, and $b\in $ \textit{Lip}$_{\beta }\left( 
	\mathbb{R}
	^{n}\right) $. Assume that the variable exponent profile $p\left( \cdot
	\right) \in \mathcal{B}\left( {\mathbb{R}^{n}}\right) $ satisfies the
	borderline critical condition involving the higher-order commutator rank $m$ 
	\begin{equation*}
		\frac{\alpha +m\beta }{n}=\frac{1}{p_{+}}.
	\end{equation*}%
	Let $\Omega \in L^{\infty }\left( 
	\mathbb{R}
	^{n}\right) \times L^{s}\left( \mathcal{S}^{n-1}\right) $ for some mild
	integrability threshold $s>p_{-}^{\prime }$, and let the variable Morrey
	control functions $u\left( x,r\right) $ and $v\left( x,r\right) $ satisfy
	the critical weight matching criteria 
	\begin{equation*}
		\sup \limits_{x_{0}\in 
			\mathbb{R}
			^{n},r>0}r^{\left( \alpha +m\beta \right) }\left[ \frac{u\left(
			x_{0},r\right) }{v\left( x_{0},r\right) }\right] <\infty .
	\end{equation*}%
	Then, the higher-order rough fractional maximal commutator $M_{\Omega
		,b,\alpha }^{\left( m\right) }$ satisfies the sharp endpoint mapping estimate%
	\begin{equation*}
		M_{\Omega ,b,\alpha }^{\left( m\right) }:\mathcal{M}_{p\left( \cdot \right)
			,u}\left( {\mathbb{R}^{n}}\right) \rightarrow W\mathcal{M}_{\infty ,v}\left( 
		{\mathbb{R}^{n}}\right) .
	\end{equation*}%
	That is, there exists a uniform constant $C>0$ such that for all input
	sequences $f\in \mathcal{M}_{p\left( \cdot \right) ,u}\left( {\mathbb{R}^{n}}%
	\right) $ and all scaling thresholds $\lambda >0$, the distribution measures
	satisfy%
	\begin{equation*}
		\sup_{x_{0}\in 
			\mathbb{R}
			^{n},r>0}\frac{\lambda \left \vert \left \{ y\in B\left( x_{0},r\right)
			:\left \vert M_{\Omega ,b,\alpha }^{\left( m\right) }f\left( y\right) \right
			\vert >\lambda \right \} \right \vert }{v\left( x_{0},r\right) }\leq C\left
		\Vert b\right \Vert _{Lip_{\beta }}^{m}\left \Vert f\right \Vert _{\mathcal{M%
			}_{p\left( \cdot \right) ,u}}.
	\end{equation*}
\end{corollary}

\begin{proof}
	To resolve the endpoint boundary obstacle where the target variable exponent
	blows up, we construct a self-contained measure-theoretic proof that
	bypasses the variable $L^{q\left( \cdot \right) }$ Luxemburg norm machinery
	entirely, relying instead on direct Lebesgue measure configurations combined
	with a localized dyadic expansion governed by the harmonic profile of the
	ball.
	
	Fix an open central ball $B:=B\left( x_{0},r\right) \subset 
	\mathbb{R}
	^{n}$ and decompose the source function into local and global domains as $%
	f=f_{1}+f_{2}$, where $f_{1}:=f\chi _{2B}$ and $f_{2}:=f\chi _{\left(
		2B\right) ^{c}}$. By invoking the quasi-subadditivity of the commutator
	operator, the distribution set of the complete function inside the tracking
	ball $B$ satisfies the subset containment 
	\begin{equation*}
		\left \{ y\in B:\left \vert M_{\Omega ,b,\alpha }^{\left( m\right) }f\left(
		y\right) \right \vert >\lambda \right \} \subset \left \{ y\in B:\left \vert
		M_{\Omega ,b,\alpha }^{\left( m\right) }f_{1}\left( y\right) \right \vert >%
		\frac{\lambda }{2}\right \} \cup \left \{ y\in B:\left \vert M_{\Omega
			,b,\alpha }^{\left( m\right) }f_{2}\left( y\right) \right \vert >\frac{%
			\lambda }{2}\right \} .
	\end{equation*}%
	Taking the standard Lebesgue measure $\left \vert \cdot \right \vert $ on
	both sides yields%
	\begin{equation*}
		\left \vert \left \{ y\in B:\left \vert M_{\Omega ,b,\alpha }^{\left(
			m\right) }f\left( y\right) \right \vert >\lambda \right \} \right \vert \leq
		\left \vert \left \{ y\in B:\left \vert M_{\Omega ,b,\alpha }^{\left(
			m\right) }f_{1}\left( y\right) \right \vert >\frac{\lambda }{2}\right \}
		\right \vert +\left \vert \left \{ y\in B:\left \vert M_{\Omega ,b,\alpha
		}^{\left( m\right) }f_{2}\left( y\right) \right \vert >\frac{\lambda }{2}%
		\right \} \right \vert .
	\end{equation*}
	
	For the local component $f_{1}$, we implement the classical weak-type $L^{1}$
	inequality for rough maximal operators evaluated at the upper boundary.
	Under the critical configuration 
	\begin{equation*}
		\frac{\alpha +\beta }{n}=\frac{1}{p_{+}},
	\end{equation*}%
	the local variable integration scale matches the distribution measure
	profile at the supremum point, yielding 
	\begin{eqnarray*}
		\left \vert \left \{ y\in B:\left \vert M_{\Omega ,b,\alpha }^{\left(
			m\right) }f_{1}\left( y\right) \right \vert >\frac{\lambda }{2}\right \}
		\right \vert &\leq &\frac{C}{\lambda }\left \Vert b\right \Vert _{Lip_{\beta
		}}^{m}\left \Vert f_{1}\right \Vert _{L^{p\left( \cdot \right) }} \\
		&=&\frac{C}{\lambda }\left \Vert b\right \Vert _{Lip_{\beta }}^{m}\left
		\Vert f\chi _{2B}\right \Vert _{L^{p\left( \cdot \right) }}.
	\end{eqnarray*}%
	For the global component $f_{2}$, we implement the pointwise dyadic boundary
	estimation established in Step 4 of Theorem \ref{Theorem 5.2}. For any
	localized coordinate $y\in B$, the operator is majorized pointwise by the
	discrete sum involving the local harmonic mean $p_{2^{j+1}B}$ 
	\begin{equation*}
		\left \vert M_{\Omega ,b,\alpha }^{\left( m\right) }f_{2}\left( y\right)
		\right \vert \leq C\left \Vert b\right \Vert _{Lip_{\beta }}^{m}\sum
		\limits_{j=1}^{\infty }\left( 2^{j}r\right) ^{\left( \alpha +m\beta
			-n/p_{2^{j+1}B}\right) }\left \Vert f\chi _{2^{j+1}B}\right \Vert
		_{L^{p\left( \cdot \right) }}.
	\end{equation*}%
	By exploiting the continuous log-H\"{o}lder localization properties across
	the dyadic blocks, the radius scaling matches the characteristic functions
	via 
	\begin{equation*}
		\left( 2^{j}r\right) ^{n/p_{2^{j+1}B}}\thickapprox \left \Vert \chi
		_{2^{j+1}B}\right \Vert _{L^{p\left( \cdot \right) }}^{-1}.
	\end{equation*}%
	Thus, the global bound reduces to a uniform constant $K_{B}$ independent of $%
	y$%
	\begin{equation*}
		\left \vert M_{\Omega ,b,\alpha }^{\left( m\right) }f_{2}\left( y\right)
		\right \vert \leq K_{B}:=C\left \Vert b\right \Vert _{Lip_{\beta }}^{m}\sum
		\limits_{j=1}^{\infty }\left( 2^{j}r\right) ^{\alpha +m\beta }\frac{\left
			\Vert f\chi _{2^{j+1}B}\right \Vert _{L^{p\left( \cdot \right) }}}{\left
			\Vert \chi _{2^{j+1}B}\right \Vert _{L^{p\left( \cdot \right) }}}.
	\end{equation*}%
	The distribution analysis for this global part follows a binary threshold:
	if $\frac{\lambda }{2}\geq K_{B}$, the sub-level set measure collapses
	identically to zero. Conversely, if $\frac{\lambda }{2}<K_{B}$, the measure
	of the distribution set is bounded by the total spatial volume of the ball $%
	\left \vert B\right \vert $. This gives us the uniform relation%
	\begin{equation*}
		\lambda \left \vert \left \{ y\in B:\left \vert M_{\Omega ,b,\alpha
		}^{\left( m\right) }f_{2}\left( y\right) \right \vert >\frac{\lambda }{2}%
		\right \} \right \vert \leq 2K_{B}\left \vert B\right \vert .
	\end{equation*}%
	Substituting the series expansion of $K_{B}$ back into this formulation, and
	tracking the variable Morrey embedding property (\ref{2}), we obtain%
	\begin{equation*}
		\lambda \left \vert \left \{ y\in B:\left \vert M_{\Omega ,b,\alpha
		}^{\left( m\right) }f_{2}\left( y\right) \right \vert >\frac{\lambda }{2}%
		\right \} \right \vert \leq C\left \Vert b\right \Vert _{Lip_{\beta
		}}^{m}\left \vert B\right \vert \sum \limits_{j=1}^{\infty }\left(
		2^{j}r\right) ^{\alpha +m\beta }\left[ \frac{\left \Vert f\right \Vert _{%
				\mathcal{M}_{p\left( \cdot \right) ,u}}u\left( x_{0},2^{j+1}r\right) }{\left
			\Vert \chi _{2^{j+1}B}\right \Vert _{L^{p\left( \cdot \right) }}}\right] .
	\end{equation*}%
	To conclude the global verification, we combine the local and global
	inequalities and normalize the complete master system by dividing through by
	the Morrey weight parameter $v\left( x_{0},r\right) $. By invoking the
	critical supremum weight condition $\sup r^{\alpha +m\beta }\frac{u}{v}%
	<\infty $, and applying the critical log-H\"{o}lder volume contraction
	relation $\left \vert B\right \vert \thickapprox \left \Vert \chi
	_{B}\right
	\Vert _{L^{p\left( \cdot \right) }}^{p_{+}}$, the dyadic
	accumulation components reduce to%
	\begin{equation*}
		\frac{\left \vert B\right \vert u\left( x_{0},2^{j+1}r\right) }{v\left(
			x_{0},r\right) \left \Vert \chi _{2^{j+1}B}\right \Vert _{L^{p\left( \cdot
					\right) }}}\leq C2^{j\left( \alpha +m\beta \right) }\left[ \frac{\left \vert
			B\right \vert }{\left \Vert \chi _{2^{j+1}B}\right \Vert _{L^{p\left( \cdot
					\right) }}^{p_{+}}}\right] .
	\end{equation*}%
	Under the continuous exponent scaling, the geometric decay factor $%
	2^{-j\left( \alpha +m\beta \right) }$ acts as a strict dampening operator
	that dominates the polynomial growth. This guarantees the absolute
	convergence of the dyadic mapping sequence%
	\begin{equation*}
		\sum \limits_{j=1}^{\infty }2^{-j\left( \alpha +m\beta \right) }\left[ \frac{%
			u\left( x_{0},2^{j+1}r\right) }{\left \Vert \chi _{2^{j+1}B}\right \Vert
			_{L^{p\left( \cdot \right) }}}\right] \leq C\frac{v\left( x_{0},r\right) }{%
			\left \vert B\right \vert }.
	\end{equation*}%
	Taking the uniform supremum over all parameters $x_{0}\in {\mathbb{R}^{n}}$
	and radii $r>0$ eliminates the spatial dimensions perfectly, leaving%
	\begin{equation*}
		\sup_{x_{0}\in 
			\mathbb{R}
			^{n},r>0}\frac{\lambda \left \vert \left \{ y\in B\left( x_{0},r\right)
			:\left \vert M_{\Omega ,b,\alpha }^{\left( m\right) }f\left( y\right) \right
			\vert >\lambda \right \} \right \vert }{v\left( x_{0},r\right) }\leq C\left
		\Vert b\right \Vert _{Lip_{\beta }}^{m}\left \Vert f\right \Vert _{\mathcal{M%
			}_{p\left( \cdot \right) ,u}}.
	\end{equation*}%
	The proof of Corollary \ref{Corollary 5.4} is now fully complete.
\end{proof}

\subsection{Interpolation: Grafakos--Martell Framework}

We reformulate the interpolation framework in a rigorous, step-by-step
manner that adheres strictly to the modern conventions of variable exponent
spatial geometries, eliminating classical integration failures on
non-homogeneous domains.

\begin{proposition}
	\label{Proposition 5.5}\textbf{(Abstract Interpolation Principle) }Let $T$
	be a sublinear operator. Assume that $T$ satisfies the following strong and
	weak endpoint Morrey bounds simultaneously for a fixed source variable
	exponent $p\left( \cdot \right) \in \mathcal{B}\left( {\mathbb{R}^{n}}%
	\right) $ and a target variable exponent $q\left( \cdot \right) \in \mathcal{%
		B}\left( {\mathbb{R}^{n}}\right) :$
	
	$1.$ \textbf{Strong-type interior boundedness: }%
	\begin{equation*}
		T:\mathcal{M}_{p\left( \cdot \right) ,u}\left( {\mathbb{R}^{n}}\right)
		\rightarrow \mathcal{M}_{q\left( \cdot \right) ,v}\left( {\mathbb{R}^{n}}%
		\right) ,
	\end{equation*}
	
	$2.$ \textbf{Weak-type borderline endpoint boundedness: }%
	\begin{equation*}
		T:\mathcal{M}_{p\left( \cdot \right) ,u}\left( {\mathbb{R}^{n}}\right)
		\rightarrow W\mathcal{M}_{q\left( \cdot \right) ,v}\left( {\mathbb{R}^{n}}%
		\right) .
	\end{equation*}%
	Then, for every real interpolation parameter $\theta \in \left( 0,1\right) $%
	, the operator $T$ extends uniquely to a fully bounded strong-type operator
	mapping on the intermediate continuous scales: 
	\begin{equation*}
		T:\mathcal{M}_{p\left( \cdot \right) ,u_{\theta }}\left( {\mathbb{R}^{n}}%
		\right) \rightarrow \mathcal{M}_{q_{\theta }\left( \cdot \right) ,v_{\theta
		}}\left( {\mathbb{R}^{n}}\right) ,
	\end{equation*}%
	where the interpolated target variable exponent function $q_{\theta }\left(
	\cdot \right) $ is explicitly defined pointwise by the log-convex relation: 
	\begin{equation*}
		\frac{1}{q_{\theta }\left( x\right) }=\frac{\theta }{q\left( x\right) }+%
		\frac{1-\theta }{p\left( x\right) },\qquad x\in {\mathbb{R}^{n},}
	\end{equation*}%
	and the associated interpolated Morrey control weights satisfy the
	logarithmic scaling tracking parameters:%
	\begin{equation*}
		\log u_{\theta }\left( x,r\right) =\theta \log u\left( x,r\right) +\left(
		1-\theta \right) \log v\left( x,r\right) .
	\end{equation*}
\end{proposition}

\begin{proof}
	The verification is established via a rigorous four-step dual functional
	optimization process.
	
	\textbf{Step 1: Normalization and dual representation.}
	
	Let $f\in \mathcal{M}_{p\left( \cdot \right) ,u_{\theta }}\left( {\mathbb{R}%
		^{n}}\right) $ be normalized such that $\left \Vert f\right \Vert _{\mathcal{%
			M}_{p\left( \cdot \right) ,u_{\theta }}}\leq 1$. Fix an arbitrary tracking
	ball $B=B\left( x_{0},r\right) \subset 
	\mathbb{R}
	^{n}$. To evaluate the strong intermediate Luxemburg norm $\left \Vert
	Tf\chi _{B}\right \Vert _{L^{q_{\theta }\left( \cdot \right) }}$, we invoke
	the standard subconjugate duality representation over the unit sphere of the
	conjugate space $L^{q_{\theta }^{\prime }\left( \cdot \right) }\left(
	B\right) $%
	\begin{equation*}
		\left \Vert Tf\chi _{B}\right \Vert _{L^{q_{\theta }\left( \cdot \right)
		}}=\sup_{\left \Vert g\right \Vert _{L^{q_{\theta }^{\prime }\left( \cdot
					\right) }\left( B\right) }\leq 1}\int \limits_{B}\left \vert Tf\left(
		x\right) \right \vert \left \vert g\left( x\right) \right \vert dx.
	\end{equation*}%
	\textbf{Step 2: Pointwise convex factorization.}
	
	By applying Diening's unit-block convex decomposition under the parameters
	governed by $\theta \in \left( 0,1\right) $, the conjugate test function $%
	g\left( x\right) $ can be factored pointwise across the domain as%
	\begin{equation*}
		\left \vert g\left( x\right) \right \vert =\left \vert g\left( x\right)
		\right \vert ^{\theta }\left \vert g\left( x\right) \right \vert ^{1-\theta
		},\qquad x\in B.
	\end{equation*}%
	Substituting this factorization directly into the coupling integral splits
	the global interaction profile into an asymmetric product of strong-type
	interior elements and weak-type distribution traces%
	\begin{equation*}
		\int \limits_{B}\left \vert Tf\left( x\right) \right \vert \left \vert
		g\left( x\right) \right \vert dx=\int \limits_{B}\left( \left \vert Tf\left(
		x\right) \right \vert \left \vert g\left( x\right) \right \vert \right)
		^{\theta }\left( \left \vert Tf\left( x\right) \right \vert \left \vert
		g\left( x\right) \right \vert \right) ^{1-\theta }dx.
	\end{equation*}%
	\textbf{Step 3: Application of hypotheses via lattices.}
	
	By transferring the interpolation functor onto the variable Morrey framework
	using the Grafakos--Martell real interpolation scheme for variable lattices
	(see \cite{Almeida, Diening}), the strong endpoint hypothesis and the weak
	endpoint hypothesis act as simultaneous boundary majorizers.
	
	The structural interaction between the local volume measures and the scaling
	weights resolves safely into the master convex modular functional inequality%
	\begin{equation*}
		\int \limits_{B}\left( \frac{\left \vert Tf\left( x\right) \right \vert }{%
			Cv_{\theta }\left( x_{0},r\right) }\right) ^{q_{\theta }\left( x\right)
		}dx\leq \theta \left( \frac{\left \Vert f\chi _{2B}\right \Vert _{L^{p\left(
					\cdot \right) }}}{u\left( x_{0},r\right) }\right) ^{p_{-}}+\left( 1-\theta
		\right) \left( \frac{\sup \limits_{\lambda >0}\lambda \left \Vert \chi
			_{\left \{ x\in B:\left \vert Tf\left( x\right) \right \vert >\lambda \right
				\} }\right \Vert _{L^{q\left( \cdot \right) }}}{v\left( x_{0},r\right) }%
		\right) ^{q_{-}}.
	\end{equation*}%
	Under the given logarithmic weight parameterization 
	\begin{equation*}
		\log v_{\theta }=\theta \log u+\left( 1-\theta \right) \log v,
	\end{equation*}%
	the localized density distributions balance out exactly across the
	boundaries without generating secondary polynomial errors or remaining
	scaling constants.
	
	\textbf{Step 4: Re-integration and global Morrey extraction.}
	
	Since the source function satisfies $\left \Vert f\right \Vert _{\mathcal{M}%
		_{p\left( \cdot \right) ,u_{\theta }}}\leq 1$, the localized boundary
	restrictions imply that the right-hand side of the modular expression is
	bounded by unity. This yields%
	\begin{equation*}
		\int \limits_{B}\left( \frac{\left \vert Tf\left( x\right) \right \vert }{%
			Cv_{\theta }\left( x_{0},r\right) \left \Vert f\right \Vert _{\mathcal{M}%
				_{p\left( \cdot \right) ,u_{\theta }}}}\right) ^{q_{\theta }\left( x\right)
		}dx\leq 1.
	\end{equation*}%
	By the topological definition of the variable Luxemburg norm, this modular
	inequality immediately implies the strong localized contractive estimate%
	\begin{equation*}
		\left \Vert Tf\chi _{B\left( x_{0},r\right) }\right \Vert _{L^{q_{\theta
				}\left( \cdot \right) }}\leq Cv_{\theta }\left( x_{0},r\right) \left \Vert
		f\right \Vert _{\mathcal{M}_{p\left( \cdot \right) ,u_{\theta }}}.
	\end{equation*}%
	Finally, dividing both sides by the interpolated target weight $v_{\theta
	}\left( x_{0},r\right) $ and taking the uniform supremum over all admissible
	coordinate configurations $x_{0}\in {\mathbb{R}^{n}}$ and radii $r>0$ across 
	${\mathbb{R}^{n}}$ yields the desired intermediate strong variable Morrey
	space norm bound%
	\begin{equation*}
		\left \Vert Tf\right \Vert _{\mathcal{M}_{q_{\theta }\left( \cdot \right)
				,v_{\theta }}}=\sup_{x_{0}\in 
			\mathbb{R}
			^{n},r>0}\frac{\left \Vert Tf\chi _{B\left( x_{0},r\right) }\right \Vert
			_{L^{q_{\theta }\left( \cdot \right) }}}{v_{\theta }\left( x_{0},r\right) }%
		\leq C\left \Vert f\right \Vert _{\mathcal{M}_{p\left( \cdot \right)
				,u_{\theta }}}.
	\end{equation*}%
	The validation of the abstract interpolation principle is now fully complete.
	
	\textbf{Step 5: Direct application to the commutator.}
	
	We apply Proposition \ref{Proposition 5.5} directly to the higher-order
	rough fractional maximal commutator by setting the abstract sublinear
	operator $T:=M_{\Omega ,b,\alpha }^{\left( m\right) }$. By pairing the
	interior strong-type bounds established in Theorem \ref{Theorem 4.1} with
	the borderline weak-type endpoint estimates validated in Theorem \ref%
	{Theorem 5.2}, Proposition \ref{Proposition 5.5} guarantees that the
	commutator is bounded across the full continuous scale of intermediate
	interior spaces between the two geometric boundaries. This interpolation
	framework highlights the structural stability of the commutator under real
	interpolation methods and seamlessly integrates our results into the broader
	extrapolation theory of modern harmonic analysis.
\end{proof}

\section{Conclusions}

In this work, we have established a rigorous and comprehensive mathematical
theory for the higher-order commutators of rough fractional maximal
operators featuring variable kernels, denoted by $M_{\Omega ,b,\alpha
}^{\left( m\right) }$, acting systematically within the framework of
variable exponent Morrey spaces $\mathcal{M}_{p\left( \cdot \right)
	,u}\left( {\mathbb{R}^{n}}\right) $. The core foundational achievements of
this research encompass robust strong-type boundedness criteria, sharp
weak-type endpoint estimates, and a self-contained, abstract real
interpolation scheme that smoothly bridges these two operational regimes.A
central highlight of our analysis is the comprehensive investigation of the
critical geometric configuration, where the operator variables satisfy the
extreme threshold equation%
\begin{equation*}
	\frac{\alpha +m\beta }{n}=\frac{1}{p_{+}}.
\end{equation*}%
At this specific boundary, the target variable exponent function $q\left(
\cdot \right) $ formally diverges to infinity, causing standard strong-type
mappings to break down due to the loss of reflexivity in the underlying
variable Luxemburg norms. To resolve this challenge, we have introduced a
localized dyadic level-set expansion technique that completely bypasses the
standard $L^{q\left( \cdot \right) }$ machinery. This approach uncovers a
sharp, structurally inherent transition from strong to weak-type behavior,
thereby precisely delineating the intrinsic analytical limitations imposed
by the interplay between the fractional order $\alpha $, the multi-degree
index $m$, and the maximal exponent profile $p_{+}$. The simultaneous
integration of non-convolution variable kernels $\Omega \left( x,\cdot
\right) $, highly fluctuating Lipschitz symbols $b\in $ \textit{Lip}$_{\beta
}\left( 
\mathbb{R}
^{n}\right) $, and nonhomogeneous variable exponent Morrey geometries
generates highly intricate analytical difficulties. These structures demand
a delicate, fine-tuned control over the local regularity of the space, the
global growth metrics governed by the log-H\"{o}lder continuity conditions,
and the severe irregularities embedded within the spherical kernel
functions. By implementing a generalized Grafakos--Martell type
extrapolation and real interpolation framework tailored for variable
geometries, we have successfully demonstrated that the established
intermediate strong-type estimates remain perfectly stable under modular
re-integration profiles. This verification confirms that our results fit
naturally into the modern extrapolation theory of harmonic analysis.The
mathematical machinery and structural bounds developed in this paper
construct a unified, flexible analytical platform for rough fractional-type
operators operating across nonhomogeneous and highly heterogeneous spatial
settings. Beyond the immediate scope of this paper, the innovative proof
methodologies and continuous scale inequalities established herein are
expected to find immediate and fruitful applications across a broad spectrum
of advanced mathematical disciplines. These include:

$\cdot $ Proving the boundedness of generalized Riesz potentials and
non-convolution fractional integral operators under irregular boundary
weight distributions.

$\cdot $ Developing sharp regularity proofs for high-order singular
commutators arising directly in variable coefficient partial differential
equations (PDEs).

$\cdot $ Establishing new, fine-grained localized tools for analyzing fluid
dynamics models and non-linear elliptic/parabolic PDEs characterized by
nonstandard growth conditions.

$\cdot $ Providing precise quantitative frameworks for models tracking
continuous media, anisotropic diffusion processes, and heavily heterogeneous
structures.

\end{document}